%% file: poro_final_final.tex
\documentclass[12pt,twoside,reqno,psamsfonts]{amsart}

\usepackage{verbatim}

\usepackage[OT1]{fontenc}
\usepackage{type1cm}
\usepackage{amssymb}
\usepackage{mathtools}

\usepackage{amsthm}
\RequirePackage{amsmath,
amsfonts,amssymb,amsthm}
\RequirePackage[dvips]{graphicx}
\usepackage{psfrag}
\usepackage[usenames]{color}
\usepackage[dvips]{graphicx} % for gfx
\usepackage{epstopdf}
\usepackage{tikz}
\usetikzlibrary{trees}
\newcommand{\myTransformation}[2]
{
    \pgftransformcm{1}{0}{0}{1}{\pgfpoint{#1cm}{#2cm}}
}

\newcommand{\myGlobalTransformation}[2]
{
  \pgftransformcm{1}{0}{0.4}{0.5}{\pgfpoint{#1cm}{#2cm}}
}

\newcommand{\gridThreeD}[4]
{
  \begin{scope}
    \myGlobalTransformation{#1}{#2};
    \draw[white,line width=3pt,opacity=1.0,step=#4cm] grid (8,8);
    \draw [#3,step=#4cm] grid (8,8);
  \end{scope}
}

\tikzstyle myBG=[line width=3pt,opacity=1.0]

\newcommand{\drawLinewithBG}[2]
{
  \draw[white,myBG]  (#1) -- (#2);
  \draw[black,very thick] (#1) -- (#2);
}
\numberwithin{equation}{section}

\theoremstyle{plain}
\newtheorem{theorem}{Theorem}[]
\newtheorem{Retheorem}{Theorem}[]

\newtheorem{proposition}[theorem]{Proposition}
\newtheorem{lemma}[theorem]{Lemma}

\theoremstyle{remark}
\newtheorem{remark}[theorem]{Remark}
\newtheorem{remarks}[theorem]{Remarks}

\theoremstyle{definition}

\numberwithin{equation}{section}

\newcommand{\PP}{\mathbb{P}}%probability

\newcommand{\EE}{\mathbb{E}}%expectation
\newcommand{\QQ}{\mathcal{Q}}%cube collection

\newcommand{\R}{\mathbb{R}}

\newcommand{\N}{\mathbb{N}}

\newcommand{\eps}{\varepsilon}

\DeclareMathOperator{\dimh}{dim_H}

\DeclareMathOperator{\dist}{dist}

\DeclareMathOperator{\por}{por}
\DeclareMathOperator{\upor}{\overline{por}}
\DeclareMathOperator{\lpor}{\underline{por}}

\DeclareMathOperator{\interior}{int}
\newcommand{\restr}[1]{\lower3pt\hbox{$|_{#1}$}}

\begin{document}

\title{Fractal percolation, porosity, and dimension}

\author{Changhao Chen}
\author{Tuomo Ojala}
\author{Eino Rossi}
\author{Ville Suomala}
 \address{Department of Mathematical Sciences\\
University of Oulu\\
PO Box 3000\\
FI-90014 Oulu\\
Finland}
\email{changhao.chen@oulu.fi}
\email{ville.suomala@oulu.fi}

 \address{University of Jyvaskyla\\
Department of Mathematics and Statistics \\
          P.O. Box 35 (MaD) \\
          FI-40014 University of Jyvaskyla \\
          Finland}
          \email{tuomo.j.ojala@jyu.fi}
        \email{eino.rossi@jyu.fi}

\thanks{CC and ER acknowledge the support of the Vilho, Yrj\"o, and Kalle V\"ais\"al\"a foundation}
\subjclass[2000]{Primary 60J80; Secondary 28A80, 60D05.}
\keywords{fractal percolation, porosity, Galton-Watson process}
\date{\today}

\begin{abstract}
  We study the porosity properties of fractal percolation sets $E\subset\R^d$. Among other things, for all $0<\varepsilon<\tfrac12$, we obtain dimension bounds for  the set of exceptional points where the upper porosity of $E$ is less than $\tfrac12-\varepsilon$, or the lower porosity is larger than $\varepsilon$. Our method works also for inhomogeneous fractal percolation and more general random sets whose offspring distribution gives rise to a Galton-Watson process. 
\end{abstract}

\maketitle
\section{Introduction}
Let $A \subset \R^d$ and $x \in A$. The porosity of $A$ at $x$ on scale $r$ is 
\begin{align*}
  \por(A,x,r) = \sup\{\varrho\geq 0 \,:\,  \text{there is a ball } B(y,\varrho r) \subset B(x,r) \backslash A\}\,,
\end{align*}
where $B(x,r)$ denotes the closed ball with center $x$ and radius $r$.
The upper and lower porosities of $A$ at $x$ are defined respectively as
$$\upor(A,x)=\underset{r\rightarrow 0}{\limsup}\por(A,x,r) \text{ and }
\lpor(A,x)=\underset{r\rightarrow 0}{\liminf}\por(A,x,r).$$ 
In analysis and geometry, sets satisfying porosity conditions like $\upor(A,x)<\alpha<\tfrac12$ for all $x\in A$ are often considered as ``small exceptional sets''. The survey \cite{Zajicek1987} is a good source of information regarding the origins of the notion of porosity and on the early achievements. In geometric measure theory, problems relating various notions of porosity to dimension have been under intensive study, see e.g. \cite{Shmerkin2011} and the references therein. More genrally, the porosity of a set gives quantitative information about the local geometry of the set and thus, it is of interest to determine the porosity properties of a given set or a family of sets.

Our ultimate goal is to study the porosity properties of a well known family of random sets known as \emph{fractal percolation} or \emph{Mandelbrot percolation}. We start by informal description of  the model. Let $d\in\N$, $2\le M\in\N$ and $0<p<1$. Let $\mathcal{Q}_n$ denote the family of closed $M$-adic sub-cubes of $[0,1]^d$ of side-length $\ell(Q)=M^{-n}$. Let $E_0=[0,1]^d$.
To define $E_1$, we choose each cube $Q\in\mathcal{Q}_1$ with probability $p$ and remove it with probability $1-p$, all choices being independent of each other. Let $E_1$ be the union of the chosen cubes. For each of the chosen cubes, we continue inductively in the same manner by dividing them into $M^d$ cubes in $\mathcal{Q}_2$ and keeping each of these with probability $p$ and removing otherwise, with all the choices independent of each other and the previous step. The union of all the chosen cubes forms the set $E_2$. Each $E_n$, $n\ge 3$ is defined inductively in the same fashion. \emph{The fractal percolation limit set} is then defined as
\[E=\bigcap_{n\in\N}E_n\,.\]

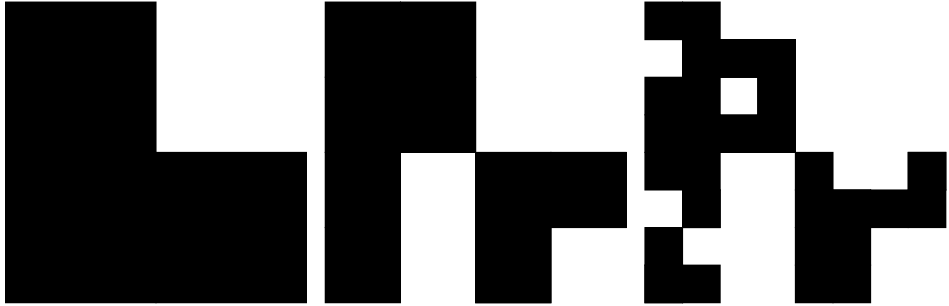
\begin{figure}
  \begin{centering}
    \begin{tikzpicture}[scale=.5]
     %draws the tree
      %first level
  
      %unfortunately the levels have to be runned in reverse order once more to get the 3d efect correctly
%third level and connections to second
      \begin{scope}
        \myTransformation{17}{0};
         \foreach \x in {3,1} {
            \foreach \y in {7,5} {
              \filldraw[black] (\x,\y)++(-1,-1) rectangle ++(2,2);
              \myTransformation{8.5}{0}
               \foreach \u in {0,1} {
                \foreach \o in {0,1} {
                 \pgfmathsetmacro{\rX}{random(0,1)}
                  \pgfmathsetmacro{\rY}{random(0,1)}
                    \filldraw [black] (\x+\rX,\y+\rY)++(-1,-1) rectangle ++(1,1);                 }
              }
               %this way circle of nodes will not be transformed
            }
          }
          \foreach \x in {1} {
            \foreach \y in {3,1} {
              \filldraw[black] (\x,\y)++(-1,-1) rectangle ++(2,2);
              \myTransformation{8.5}{0}
              \foreach \u in {0,1} {
                \foreach \o in {0,1} {
                  \pgfmathsetmacro{\rX}{random(0,1)}
                  \pgfmathsetmacro{\rY}{random(0,1)}
                    \filldraw [black] (\x+\rX,\y+\rY)++(-1,-1) rectangle ++(1,1); 
                }
              }

                %this way circle of nodes will not be transformed
            }
          }
          \foreach \x in {7,5} {
            \foreach \y in {3} {
              \filldraw[black] (\x,\y)++(-1,-1) rectangle ++(2,2);
              \filldraw[black] (5,1)++(-1,-1) rectangle ++(2,2);
              \myTransformation{8.5}{0}
              \foreach \u in {0,1} {
                \foreach \o in {0,1} {
                \pgfmathsetmacro{\rX}{random(0,1)}
                  \pgfmathsetmacro{\rY}{random(0,1)}
                    \filldraw [black] (\x+\rX,\y+\rY)++(-1,-1) rectangle ++(1,1);         
                  
                  }

              }

                %this way circle of nodes will not be transformed
            }
        }
\myTransformation{8.5}{0}
        \foreach \u in {0,1} {
          \foreach \o in {0,1} {

            \pgfmathsetmacro{\rX}{random(0,1)}
            \pgfmathsetmacro{\rY}{random(0,1)}

            \filldraw[black] (5+\rX,1+\rY)++(-1,-1) rectangle ++(1,1);
          }}
      \end{scope}

      %second level and the connections to root
      \begin{scope}
        \myTransformation{8.5}{0};
         \foreach \x in {2} {
            \foreach \y in {6,2} {
                \filldraw[black] (\x,\y)++(-2,-2) rectangle ++(4,4);
                %this way circle of nodes will not be transformed
            }
        }
                \filldraw[black] (4,0) rectangle ++(4,4);
                %this way circle of nodes will not be transformed

      \end{scope}

    \end{tikzpicture}
  \end{centering}
  \caption{The first three steps of the construction in a model where $d=2$, $M=2$.}
  \label{figure1}
\end{figure}

Obviously, this process crucially depends on the parameter $p$ (as well as $M,d$). For instance, it is well known that (see Section \ref{sub:GW} below) that $\mathbb{P}(E\neq\varnothing)>0$ if and only if $p>M^{-d}$ and in this case, 
conditioned on \emph{non-extinction} (that is $E\neq\varnothing$), the almost sure Hausdorff and box-counting dimension of $E$ equals 
\begin{equation}\label{eq:s}
  s=s(p,d,M):=d+\log p/\log M\,.
\end{equation} 
Further, there is a natural limit measure $\mu$ on $E$ defined as the weak limit of
\begin{equation}\label{eq:mudef}
  \mu_n=p^{-n}\mathcal{L}|_{E_n}\,,
\end{equation}
which exists and has exact dimension $s(p,d,M)$ a.s. on $E\neq\varnothing$.

The study of the porosity properties of fractal percolation was initiated by J\"arvenp\"a\"a, J\"arvenp\"a\"a and Mauldin \cite{Jetal2002}, who conjectured that almost surely, the upper and lower porosities should take the extremal values ($\tfrac12$ and $0$, respectively) at $\mu$-almost all points of $E$. This conjecture was later verified by Berlinkov and J\"arvenp\"a\"a \cite{BerlinkovJarvenpaa2003}, who also studied mean porosities and the porosities of the natural measure $\mu$. A related result saying that $E$ and its orthogonal projections have full Assouad dimension almost surely on non-extinction can be found in \cite{FraserMiaoTroscheit2015}. For the latter statement, it is enough that there is at least one point $x\in E$ with $\lpor(E,x)=0$.

In this paper, we use a different method via Galton-Watson branching processes to estimate the dimension of the exceptional sets, where the porosities take values other than $0$, $\tfrac12$.

The porosity of the graph of Brownian motion has been studied in \cite{CoxPhilip1991}, where an almost sure value for $\dim_H\{x\in E\,:\,\upor(E,x)=\alpha\}$ is calculated  for all $\alpha\le\tfrac12$ where $E$ is the graph of one-dimensional Brownian motion.
In particular, the results of \cite{CoxPhilip1991} indicate that the upper porosity level sets for Brownian graphs have a mutifractal structure.
Thus, our results for the upper porosity of the fractal percolation are analogous to the ones obtained in \cite{CoxPhilip1991} for the graph of Brownian motion. However in \cite{CoxPhilip1991}, the method is based on return times of the Brownian motion, whereas our results for fractal percolation limit sets rely on the properties of the underlying Galton-Watson process. Furthermore, we also consider the dimension of the sets where lower porosity takes values $\ge\varepsilon>0$.

Our main results are the following. Here $M$ and $d$ are fixed, $s=s(p,d,M)$ is as in \eqref{eq:s} and $\dim_H$ denotes Hausdorff dimension. 

\begin{theorem}
  \label{thm:E_upper_porous}
  For any $p > M^{-d}$, there exists $c=c(d,M,p)>0$, such that a.s. on non-extinction, $\underset{{x\in E}}{\inf}\upor(E,x)= c$.
\end{theorem}

\begin{theorem}
  \label{thm:upor_dimbound}
  For any $p > M^{-d}$, $\varepsilon>0$ there exists $\delta =\delta(d,M,p,\varepsilon)>0$, such that a.s. 
  \[
    \dim_H\{x\in E: \upor(E,x) < 1/2-\varepsilon\} < s -\delta.
  \]
\end{theorem}

\begin{theorem}
  \label{thm:lpor_dimbound}
  For any $p > M^{-d}$, $\varepsilon>0$ there exists $\delta =\delta(d,M,p,\varepsilon)>0$, such that a.s. 
  \[
    \dim_H\{x\in E: \lpor(E,x) > \epsilon\} < s -\delta.
  \]
\end{theorem}
Regarding lower bounds for the dimension of the points where the porosities take exceptional values, we have the following theorems.
\begin{theorem}
  \label{thm:upor_lowerbound}
  For any $p > M^{-d}$, there is $\varepsilon_0=\varepsilon_0(d,M,p)>0$ such that for $\varepsilon\in (0, \varepsilon_0)$ there exists $\delta =\delta(d,M,p,\varepsilon)>0$, such that a.s. on non-extinction,
  \[
    \dim_H\{x \in E: \upor(E,x) \leq 1/2-\varepsilon\} > \delta.
  \]
\end{theorem}

Theorem \ref{thm:upor_lowerbound} shows in particular that $c \neq 1/2$ in Theorem \ref{thm:E_upper_porous}.

\begin{theorem}
  \label{thm:lpor_lowerbound}
  For any $p > M^{-d}$, there is $\varepsilon_0=\varepsilon_0(p,M,d)>0$ such that for $\varepsilon\in (0, \varepsilon_0)$ there exists $\delta =\delta(d,M,p,\varepsilon)>0$, such that a.s. on non-extinction,
  \[
    \dim_H\{x\in E: \lpor(E, x)> \varepsilon\} > \delta.
  \]
\end{theorem}

One can read quantitative bounds for $c$ in Theorem \ref{thm:E_upper_porous} and  $\delta$ in Theorems \ref{thm:upor_dimbound}--\ref{thm:lpor_lowerbound} from the proofs. See the discussion in Section \ref{sec:remarks}.

Since the natural measure $\mu$ is a.s. exact dimensional of dimension $s$ (see \cite{MauldinWilliams86}), the Theorems \ref{thm:upor_dimbound} and \ref{thm:lpor_dimbound} generalize 
the result of Berlinkov and J\"arvenp\"a\"a concerning the upper and lower porosities of $E$ at $\mu$-almost all points.

All these results remain valid for the \emph{inhomogeneous fractal percolation}, where instead of a fixed $p$, each $Q\in\mathcal{Q}_1$ is chosen with probability $0<p_Q<1$ and the process is continued in a self-similar way (see e.g. \cite{RamsSimon2015} for a detailed description of the model). 
In fact,it is allowed that some cubes are chosen deterministically (i.e. $p_Q=1$ for some, but \textbf{not all}, $Q\in\mathcal{Q}_1$). If it is allowed that some $p_Q=0$, then Theorem \ref{thm:lpor_dimbound} obviousy fails, but the other result are still valid. In Section \ref{sec:MoreGeneral}, we will provide extensions of these results for random sets defined using more general selection processes.

In the above theorems, the actual value of the dimension is, in fact, a constant a.s. conditioned on non-extinction. The proof of 
this fact is relatively easy and follows as a simple corollary of a zero-one law 
for the percolation process, see Lemma \ref{lemma:zero-one}. 

\begin{theorem}
  For any $p > M^{-d}, 0\leq \alpha \leq \frac{1}{2}$, there exists $\beta_u=\beta_u(\alpha)$ and $\beta_l=\beta_l(\alpha)$, such that a.s. on non-extinction
  \[
    \dim_H\{x\in E: \upor(E,x) \leq \alpha\} = \beta_u
    \,\,\text{ and }\,\, 
    \dim_H\{x\in E: \lpor(E,x) \geq \alpha\} = \beta_l.
  \]
  \label{thm:asDim_upor_ubound}
\end{theorem}

Thus, in fact the Theorems \ref{thm:Re_upor_dimbound}--\ref{thm:Re_lpor_lowerbound} just give estimates for $\beta_u$ and $\beta_l$.   Note that $\dim_H\{x\in E: \lpor(E,x) \leq \alpha\}$ and $\dim_H\{x\in E: \upor(E,x) \geq \alpha\}$, equal $s$ almost surely,  since at
$\mu$-almost every point, $E$ has minimal lower porosity and maximal upper porosity.  

The structure of the paper is as follows. In Section \ref{sec:prel}, we will set up some notation and recall some basic results related to Galton-Watson branching processes and Galton-Watson trees. We will also explain how these results may be applied to families of nested random sets in $\R^d$ whose offspring distribution gives rise to a Galton-Watson tree. Our main results on the porosity properties of fractal percolation limit sets are proved in Section \ref{sec:proofs}. In Section \ref{sec:MoreGeneral}, we extend these results for the more general Galton-Watson type random sets mentioned above. We conclude in Section \ref{sec:remarks} with additional remarks and open problems.  

\section{Preliminaries}\label{sec:prel}

\subsection{Trees and subtrees}
A \emph{Tree} $T$ is a connected graph without cycles and with a distinguished vertex $\rho$ called the \emph{root}. 
If $v$ is a vertex of the tree, we let $|v|$ denote the number of edges on the shortest path from $\rho$ to $v$. In this case, we also say that ``$v$ is at distance $n$ from the root''. For $v\neq \rho$, 
we denote its immediate ancestor by $v^{-}$. For vertices $v$ and $u$, let
$v\wedge u$ be the last common vertex on the shortest paths from $\rho$ to $v$ and $u$. Denote $v\ge u$, if the shortest path connecting $\rho$ to $u$ visits the vertex $v$.

\emph{A Ray} is an infinite path $\sigma=v_1v_2\ldots$ on the tree that never visits a vertex twice. We define the (reduced) boundary of the tree to consist of all rays starting from the root and denote this by $\partial T$. The notations $\sigma\wedge\tau$ and $v\le\sigma$ are naturally extended for $\sigma,\tau$ in the boundary.

Given $0<\lambda<1$, we may define a metric on $\partial T$ as
\begin{equation*}
  d_\lambda(\sigma,\tau)=
  \begin{cases}
    \lambda^{|\sigma\wedge\tau|}&\text{ if }\sigma\neq\tau,\\
    0&\text{ if }\sigma=\tau\,.
  \end{cases}
\end{equation*}

\subsection{Galton-Watson branching processes and Galton-Watson trees}
\label{sub:GW}

The Galton-Watson ($GW$ in what follows) branching processes and Galton-Watson trees will be the most essentials concepts in our proofs. We next recall their definition and some known facts (see e.g. \cite{Peres1999}, \cite{LyonsPeres2014}). 
Let $L$ be a non-negative integer valued random variable 
and $\{L_{n,i}\}_{n,i \in \N}$ a sequence of independent copies of $L$. Let $Z_0=1$, $Z_1=L$ and
\[
  Z_{n+1}=\begin{cases}
    \sum^{Z_n}_{i=1} L_{n,i},\textrm{ if } Z_n >0\\
    0, \textrm{ if } Z_n=0
  \end{cases}
\]
for $n\in\N$.
This procedure defines a Markov chain $(Z_n)_{n\ge 0}$, which we call the \emph{Galton-Watson process with offspring distribution $L$}. There is a natural way to grow trees at random using a given $GW$-process; The number of vertices at distance $n\ge 1$ from the root equals $Z_n$ and if these vertices are denoted $v_{n,1},\ldots v_{n,Z_n}$, then $v_{n,k}$ has $L_{n,k}$ children. These trees are termed \emph{Galton-Watson trees}.

The most basic question in the study of branching processes is to find the value of the extinction probability 
\begin{equation}\label{eq:extinction}
  q = \PP( \text{eventually } Z_n =0 )\,,  
\end{equation}
and in particular, whether this is strictly less than one.
For $GW$-processes, it is well known that (assuming $\PP(L=1)<1$) 
\begin{equation}\label{eq:criticalp}
  q<1\text{ if and only if }\EE(L)>1\,.
\end{equation}

Conditional on non-extinction, what can be said about the size of $\partial T$ for a $GW$-tree $T$? In terms of dimension, there is a satisfactory answer available: Almost surely on non-extinction,
\begin{equation}\label{eq:boundarydim}
  \dim_H(\partial T)=\dim_B(\partial T)=\frac{\log\EE(L)}{-\log \lambda}\,.
\end{equation}
where $\dim_B$ denotes the box-counting dimension and these dimensions are calculated in the $d_\lambda$ metric as defined earlier. This result was proved by Hawkes \cite{Hawkes1981} under the assumption that $\EE[L(\log^{+} L)^2]<\infty$ and then by Lyons \cite{Lyons1990} for general $L$. We note that already Hawkes' result is enough for our purposes since in our applications to the random sets, $L$ will always be bounded.

We denote by $\QQ_n$ the family of closed $M$-adic sub-cubes of level $n$ of the unit cube $[0,1]^d$,
\[\QQ_n=\left\{\prod_{l=1}^d[i_l M^{-n},(i_{l}+1)M^{-n}]\,:\,0\le i_l\le M^{n}-1\right\}\]
and let $\QQ=\cup_{n\in\N}\QQ_n$. Write $Q'\prec_N Q$ if there is $n\in\N$ such that $Q\in\QQ_n$, $Q'\in\QQ_{n+N}$ and $Q'\subset Q$.
For later use, let $\textbf{1}[\mathcal{A}]$ denote the indicator function of an event $\mathcal{A}$ and for
$A,B\subset\R^d$, we denote $\dist(A,B)=\inf\{|x-y|\,:\,x\in A, y\in B\}$.

\subsection{Random fractals that give rise to $GW$-trees}
\label{sub:rf}

We next recall how fractal percolation is related to a branching process and a $GW$-tree. For later use, we present this connection in a more general setting.

To that end, suppose that $X_Q$, $Q\in\QQ$, $n\ge 0$, are random subsets of $\mathcal{Q}_1$ such that the number of $Q\in\mathcal{Q}_1$ contained in $X_Q$ is $L_Q$, where $L_Q$ are independent and identically distributed according to an initial random variable $L\in\{0,\ldots M^d\}$. We note that we are not assuming that $X_Q$ are independent or identically distributed, even thought $L_Q$ are.\footnote{The amount of selected sub-cubes are iid, but the way these sub-cubes are distributed inside the parent cube is basically free. Their distributions can be different for different parent cubes and all kinds of dependencies are allowed.} 
Let $E_0=[0,1]^d$ and $E_1=\cup X_{[0,1]^d}$. If $E_n=Q_1\cup\ldots\cup Q_{Z_n}$, $Q_i\in\mathcal{Q}_n$, we set
\[E_{n+1}=\bigcup_{i=1}^{Z_n}h_{Q_i}(\cup X_{Q_i})\,,\]
where $h_{Q_i}$ is the homothety (scaling composed with translation) sending $[0,1]^d$ onto $Q_i$ (and we set $E_{n+1}=\varnothing$, if $E_n=\varnothing$). Now $E_n$ is a decreasing sequence of random sets closely connected to the $GW$-process with offspring distribution $L$: The cubes in $\mathcal{Q}_n$ forming $E_n$ may be put into one to one correspondence with the vertices of the $GW$-tree $T$ at distance $n$ from the root (see Figure \ref{fig:CubesAndTrees}).
%%%%%%%%%%%%%%%%%%%%%%%%%%%%%%%%%%%%%%%%%%%%%%%%%%%%%%%%%%%%%%%%%%%%%TIKZ-PICTURE%%%%%%%%%%%%%%%%%%%%%%%%%%%%%%%%%%%%%%%%%%%%%%%%%%%%
\begin{figure}
  \begin{centering}
    \begin{tikzpicture}[scale=.7]
    %draws projection-grid:
      \gridThreeD{0}{0}{black!50}{2};
      \gridThreeD{0}{4.25}{black!50}{4};
      \gridThreeD{0}{8.50}{black!50}{8};
      %draws the tree
      %first level
      \begin{scope}
        \myGlobalTransformation{0}{8.5};
        \node (root) at (4,4) [circle,fill=black] {};
      \end{scope}
      %second level and the connections to root
      \begin{scope}
        \myGlobalTransformation{0}{4.25};
        \foreach \x in {2} {
          \foreach \y in {6,2} {
            \node (child\x\y) at (\x,\y) [circle,fill=black] {};
            \drawLinewithBG{root}{child\x\y};
                %this way circle of nodes will not be transformed
          }
        }
        \node (child62) at (6,2) [circle,fill=black] {};
        \drawLinewithBG{root}{child62};
                %this way circle of nodes will not be transformed
      \end{scope}
      %third level and connections to second
      \begin{scope}
        \myGlobalTransformation{0}{0};
        \foreach \x in {3,1} {
          \foreach \y in {7,5} {
            \node (child\x\y) at (\x,\y) [circle,fill=black] {};
            \drawLinewithBG{child26}{child\x\y};
                %this way circle of nodes will not be transformed
          }
        }
        \foreach \x in {1} {
          \foreach \y in {3,1} {
            \node (child\x\y) at (\x,\y) [circle,fill=black] {};
            \drawLinewithBG{child22}{child\x\y};
                %this way circle of nodes will not be transformed
          }
        }
        \foreach \x in {7,5} {
          \foreach \y in {3,1} {
            \node (child\x\y) at (\x,\y) [circle,fill=black] {};
            \drawLinewithBG{child62}{child\x\y};
                %this way circle of nodes will not be transformed
          }
        }
      \end{scope}
      %unfortunately the levels have to be runned in reverse order once more to get the 3d efect correctly
%third level and connections to second
      \begin{scope}
        \myGlobalTransformation{0}{0};
        \foreach \y in {7,5} {
          \foreach \x in {3,1}{
            \filldraw[black!30,opacity=.5] (\x,\y)++(-1,-1) rectangle ++(2,2);
            \node (child\x\y) at (\x,\y) [circle,fill=black] {};
            \drawLinewithBG{child26}{child\x\y};
                %this way circle of nodes will not be transformed
          }
        }
        \foreach \x in {1} {
          \foreach \y in {3,1} {
            \filldraw[black!30,opacity=.5] (\x,\y)++(-1,-1) rectangle ++(2,2);
            \node (child\x\y) at (\x,\y) [circle,fill=black] {};
            \drawLinewithBG{child22}{child\x\y};
                %this way circle of nodes will not be transformed
          }
        }
        \foreach \y in {3,1} {
          \foreach \x in {7,5}  {
            \filldraw[black!30,opacity=.5] (\x,\y)++(-1,-1) rectangle ++(2,2);
            \node (child\x\y) at (\x,\y) [circle,fill=black] {};
            \drawLinewithBG{child62}{child\x\y};
                %this way circle of nodes will not be transformed
          }
        }
      \end{scope}

      \gridThreeD{0}{4.25}{black!50}{4};
      %second level and the connections to root
      \begin{scope}
        \myGlobalTransformation{0}{4.25};
        \foreach \x in {2} {
          \foreach \y in {6,2} {
            \filldraw[black!30,opacity=.5] (\x,\y)++(-2,-2) rectangle ++(4,4);
            \node (child\x\y) at (\x,\y) [circle,fill=black] {};
            \drawLinewithBG{root}{child\x\y};
                %this way circle of nodes will not be transformed
          }
        }
        \filldraw[black!30,opacity=.5] (4,0) rectangle ++(4,4);
        \node (child62) at (6,2) [circle,fill=black] {};
        \drawLinewithBG{root}{child62};
                %this way circle of nodes will not be transformed

      \end{scope}

      \gridThreeD{0}{8.50}{black!50}{8};
      \begin{scope}
        \myGlobalTransformation{0}{8.5};
        \filldraw[black!30,opacity=.5] (0,0) rectangle ++(8,8);
        \node (root) at (4,4) [circle,fill=black] {};
      \end{scope}
    \end{tikzpicture}
  \end{centering}
  \caption{The correspondence of a random fractal and a GW -tree with $d=2$ and $M = 2$. Note that the fractal percolation process $E_n$ encodes also the geometric information, while the corresponding 
  GW-tree only contains information about the number of cubes $L_Q$.}
  \label{fig:CubesAndTrees}
\end{figure}
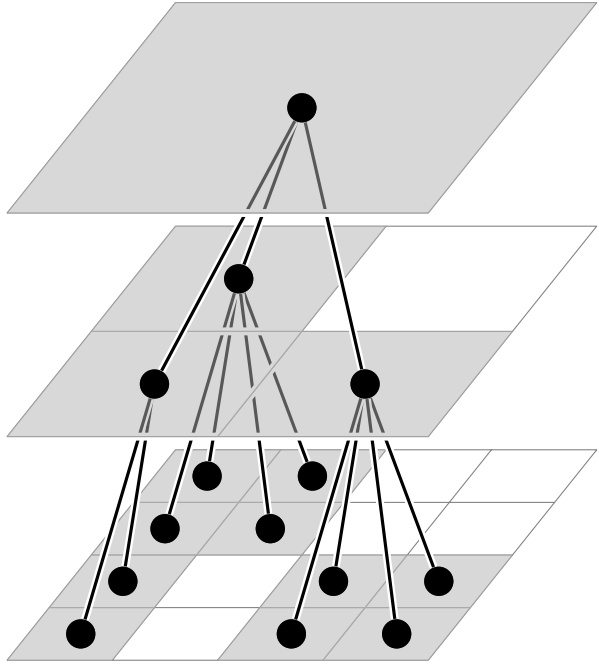
%%%%%%%%%%%%%%%%%%%%%%%%%%%%%%%%%%%%%%%%%%%%%%%%%%%%%%%%%%%%%%%%%%%%%%%%%%%%%%%%%%%%%%%%%%%%%%%%%%%%%%%%%%%%%%%%%%%%%%%%%%%%%%%%%%%%%
Further, if $v=v(Q)$, $w=w(Q')$ are vertices of this $GW$ tree corresponding to $Q\in\mathcal{Q}_n$, $Q'\in\mathcal{Q}_{n+1}$, then $v=w^-$ if and only if $Q'\subset h_{Q}(\cup X_Q)\cap E_{n+1}$. To study the properties of the random limit set $E=\cap_{n\in\N}E_n$, we define 
a natural projection $\Pi$ from $\partial T$ onto $E$. If $\sigma=v(Q_1)v(Q_2)\ldots
\in\partial T$, we set $\{\Pi(\sigma)\}=\cap_{n}Q_{n}$.

Now what does \eqref{eq:boundarydim} tell us about the random sets $E=\cap_{n \in \N} E_n$ defined above. Choosing $\lambda=M^{-1}$, it follows that $\Pi\colon\partial T\to E$ preserves Hausdorff as well as box-counting dimension, simply because on $\R^d$ these dimensions can be calculated using the $M$-adic cubes. Thus we arrive at the following important result. Almost surely on $E\neq\varnothing$,
\begin{equation}\label{eq:dim_formula}
  \dim_B(E)=\dim_H(E)=\frac{\log\mathbb{E}(L)}{\log M}\,.
\end{equation}
In the special case of fractal percolation, this implies the familiar formulas for the critical probability and for the almost sure Hausdorff dimension of $E$.

Fractal percolation (both homogeneous and inhomogeneous) are main examples of such constructions, but there are many others and we will come back to these later.  
For fractal percolation with parameter $p$ one has 
\begin{equation}\label{eq:p_k_perc}
  \PP(L=k) =
  \binom{M^d}{k}p^k(1-p)^{M^d-k}
\end{equation}
and $\mathbb{E}(L)=p M^d$.

\begin{remark}
  Formally, the law of the random set $E=E(\omega)$ is the completion of the infinite product of the discrete probability measures defining $X_Q$ for $Q\in\mathcal{Q}$ and we will denote it by $\mathbb{P}$. The random variables $\{X_Q\}_{Q\in\mathcal{Q}_n\,n\le k}$ give rise to an increasing filtration of $\sigma$-algebras $\mathcal{B}_k$ and the probability measure $\PP$ is defined on the $\sigma$-algebra $\mathcal{B}$ (on a probability space $\Omega$) induced by this filtration. In practise, we abuse notation slightly and denote e.g.
  $\PP(\upor(E,x)>c\text{ for all }x\in E)$ instead of $\PP(\omega\in\Omega\,:\,\upor(E(\omega),x)>c\text{ for all }x\in E)$.
  For practical purposes, we will also denote the law of the GW-process and the corresponding GW-tree by $\PP$, as well as the law of the discrete random variable $L_o$ generating the GW-process. This is for simplicity of notation and should not cause any confusion.
\end{remark}

\section{Proofs of Theorems \ref{thm:E_upper_porous}--\ref{thm:lpor_lowerbound}}
\label{sec:proofs}

\subsection{Outline of the method}

Our general strategy in all the proofs of Theorems \ref{thm:E_upper_porous}--\ref{thm:lpor_lowerbound} is to find a sequence  $F_n\subset E_n$ such that the exceptional set we are investigating contains (or is contained) 
in $F=\cap_{n \in \N} F_n$. We do this in such a way that the number of cubes forming $F_n$ defines a $GW$-process, so that \eqref{eq:criticalp} and  \eqref{eq:dim_formula} can be used to calculate the dimension 
(and extinction/non-extinction) of $F$, and whence in estimating the dimensional properties of the exceptional set at hand.

Although we do not formulate any of our results in terms of trees, there is a conceptual connection to problems which study the probability and existence of sufficiently regular sub-trees of $GW$-trees. See \cite{ChayesChayesDurrett1988, PakesDekking1991, LyonsPeres2014}. Roughly speaking, we look at a given $GW$-tree with a fixed parameter $N$ and label the vertices, where the induced sub-tree up to level $N$ has some required property (depending on the problem at hand). We look at the sub-tree induced by this labelling and provided $N$ is large enough, we show that this sub-tree becomes almost surely extinct or has certain dimension, etc.  

\subsection{Preparations}

For later use in Section \ref{sec:MoreGeneral} we present the following notations and definitions in the context described in Section \ref{sub:rf} above.

We introduce the following notation: We say that $Q\in\mathcal{Q}_n$ is \emph{surviving}, if for each $m\ge 1$, there is $Q_m\prec_m Q$ such that $Q_m\subset E_{n+m}$. This is essentially, but not exactly, the same as assuming $Q\cap E\neq\varnothing$. The difference is that $\partial Q\cap E$ may be nonempty even if $Q$ is not surviving, if some of the neighbouring cubes of $Q$ have surviving points in the boundary. To avoid confusion, we stress that this notation is always defined in terms of the process defining the initial random set $E$ (shortly we will consider various subsets of $E_n$ and $E$ also defined in terms of $GW$-processes). 

We will abuse notation slightly and denote simply 
\[
  \#E_n := \#\left\{Q \in \QQ_n : Q \subset E_n  \right\} 
\]
even though $E_n$ is not a collection of cubes but rather a union of such a collection.
Let us further denote by $L_o$ the offspring distribution of the $GW$-process $\# E_{n}$. In case of fractal percolation, this is given by  \eqref{eq:p_k_perc}.

Denote by $Z_n$ the number of surviving cubes in $\mathcal{Q}_{n}$ and for each $N\in\N$, consider the random variable $\widetilde{L}=\widetilde{L}^{N}=\widetilde{L}^{N,M,d,p}$ such that
\[\mathbb{P}(\widetilde{L}=k)=\mathbb{P}(Z_N=k\,|\,E\neq\varnothing)\,.\]
Since $\# E_n$ is a $GW$-process with offspring distribution $L_o$ we notice that for each $Q\in\mathcal{Q}_N$ for which $\PP(Q \subset E_N) > 0$, 
\[
  \PP\left( Q \text{ is surviving } \,|\, Q \subset E_{N} \right)= \PP\left( E \neq \varnothing \right) = 1-q \, ,
\] 
where $q$ is the extinction probability of the original process, which only depends on $L_o$. Thus
\[
  \PP\left(Q\text{ is surviving})=\PP(Q\text{ is surviving }\,|\,Q\subset E_N\right)\PP(Q\subset E_N)=(1-q) \PP(Q\subset E_N)\,,
\]
and
\begin{align*}
  &\PP\left( \#\left\{Q'\prec_N Q\,:\,Q'\text{ is surviving}\right\}  = k  \,|\, Q\subset E_{n N}  \right)\\
  &=\PP\left( \#\left\{Q' \in \QQ_{N}\,:\,Q'\text{ is surviving}  \right\} = k \right),
\end{align*}
for all $Q\in\QQ_{nN}$ with $\PP(Q\subset E_{nN})>0$.
It follows that
conditional on $E\neq\varnothing$, $(Z_{nN})_{n\ge 1}$ is a $GW$-process with offspring distribution $\widetilde{L}$.

We can now further compute the expectation of $\widetilde{L}$ as follows, 
\begin{align*}
  \EE(Z_N)&=\sum_{Q\in\mathcal{Q}_N}\PP\left(Q\text{ is surviving}\right) \\
  &=(1-q)\sum_{Q \in \QQ_N} \PP(Q \subset E_{N}) = (1-q) \EE(\# E_N) = (1-q) \EE(L_o)^N\,,
\end{align*}
where that last equation is a standard fact for $GW$-processes (see \cite{LyonsPeres2014}).
Thus, for the conditional expectation
\begin{equation}\label{eq:EL}
  \EE(\widetilde{L})=\EE(Z_N\,|\,\text{non-extinction})=(1-q)^{-1}\EE(Z_N)= \EE(L_o)^N\,.
\end{equation}

In the sequel, we will consider various $GW$-processes defined as pushforwards of $(Z_{nN})_n$ (for suitably chosen $N$) and compare their mean offspring distribution with that of $\widetilde{L}$.

Next we present a simple zero-one law for the random sets $E$.
Let $\mathcal{A}$ be a property among the pairs $Q\in\QQ$, $E\subset[0,1]^d$ such that
\begin{itemize}
  \item Whether $(Q,E)$ has the property $\mathcal{A}$ is completely determined by the random variables $\{X_{Q'}\}$ for the sub-cubes $Q'\subset Q$ (of all generations) of $Q$.
  \item If $(Q,E)$ has property $\mathcal{A}$, then $(Q',E)$ has property $\mathcal{A}$ whenever $Q\prec_N Q'$ for some $N\in\N$.
\end{itemize}
In what follows, we call such property $\mathcal{A}$ \emph{an admissible property}. This is closely connected to 
the more standard notion of inherited property among $GW$-trees.
We use this variant to allow the property to depend also on the geometry, and not only on the $GW$-tree that the percolation
process generares.

\begin{lemma}\label{lemma:zero-one}
  Suppose that $\mathcal{A}$ is an admissible property. If $\EE(L_o)>1$ and if there is $c>0$ such that $\PP((Q,E)\text{ has property }\mathcal{A}\,|\,Q\subset E_n)\ge c$ for all $Q\in\QQ_n$ and all $n\in\N$, then $\PP(([0,1]^d,E)\text{ has property }\mathcal{A}\,|\,\text{non-extinction})=1$.
\end{lemma}

\begin{proof}
  According to a basic fact for $GW$-processes, $\#E_n\longrightarrow\infty$ a.s. on non-extinction (in fact, it grows exponentially fast, see \cite[\S 5]{LyonsPeres2014}). Given $M\in\N$, let $X$ be the random variable that equals the smallest natural number $n$ with $\#E_n\ge M$. Conditional on $X=n$, $E_n$ contains at least $M$ cubes $Q\in\QQ_n$, the estimate $\PP((Q,E)\text{ has property }\mathcal{A})\ge c$ holds for each of them, and the events $(Q,E)\text{ has property }\mathcal{A}$ are independent. Thus
  \[\PP\left(([0,1]^d,E)\text{ has property }\mathcal{A}\,|\,X=n \right)\ge 1-(1-c)^M\,.\]
  Since the events $X=n$ are disjoint for different values of $n$ and their union has full probability, we get the same estimate 
  \begin{align*}
    \PP\left(([0,1]^d,E)\text{ has property }\mathcal{A} \right)\ge 1-(1-c)^M\,,
  \end{align*}
  for the unconditional probability as well. But this holds for all $M\in\N$ and the claim follows.
\end{proof}

Throughout the rest of this section, we consider fractal percolation with parameters $M,d,p$ such that $p>M^{-d}$.
In particular with the above notation
\begin{equation*}
  \EE(\widetilde{L})=\EE(L_o)^N=(p M^d)^N\,.
\end{equation*}
Recall that $s=s(d,M,p)$ denotes the almost sure dimension of $E$ conditioned on non-extinction as defined in \eqref{eq:s}.

Before going to the dimension bounds for the exceptional sets for porosity, we give the proof for Theorem \ref{thm:E_upper_porous} providing the uniform bound $\upor(E,x) \geq c$ valid for all $x\in E$.

\begin{Retheorem}
  \label{thm:Re_E_upper_porous}
  For any $p > M^{-d}$, there exist $c=c(d,M,p)>0$, such that a.s. on non-extinction, $\underset{{x\in E}}{\inf}\upor(E,x)= c$. 
\end{Retheorem}
\begin{proof}
  Our goal is to find $N=N(M,d,p)\in\N$ such that almost surely, the set 
  \[
    E^N_Q=\left\{x\in E\cap Q\,:\,\por(E,x,\sqrt{d}M^{-kN})<\frac12 d^{-1/2} M^{-N}\text{ for all }k\ge n_0\right\}
  \]
  is empty for all $Q\in\QQ_{n_0}$ and $n_0\in\N$. If this holds, then
  \[\left\{ x\in E \,:\, \upor(E,x)< \frac12d^{-1/2} M^{-N} \right\}\subset\bigcup_{n_0\in\N}\bigcup_{Q\in\mathcal{Q}_{n_0}} E_{Q}^N=\varnothing\,,\]
  and thus $\inf_{x\in E}\upor(E,x)\ge\tfrac12d^{-1/2}M^{-N}$.

  %%%%%%%%%%%%%%%%%%%%%%%%%%%%%%%%%%%%%%%%%%%%%%%%%%%%%%%%TIKZ-PICTURE%%%%%%%%%%%%%%%%%%%%%%%%%%%%%%%%%%%%%%%%%%%%

  \begin{figure}[h]
    \begin{centering}
      \begin{tikzpicture}[scale=.7]
        \draw[step=2cm,draw=black!50] grid (8,8);
        \node[right] at (8,0) {$Q$};
        \draw[fill](0,0) circle (1pt) node [below] {$x$};
        \draw (8,8) arc (45:80:7.1) node[draw=none] (loppu) {};
        \draw (8,8) arc (45:30:7.1);
        \draw[->] (0,0) -- (8,8) node[midway,right] {$\sqrt{d} M^{-k N}$};
        \draw (3,5) circle (1cm);
        \node[right] at (3,5) {$Q'$};
        \draw (3,5) -- +(100:1) node[midway,left] (smallradius) {};
        \node (teksti) at (1,4.5) {$\frac{1}{2} M^{-k N -N}$}; 
        \draw[->] (teksti) to[bend left] (smallradius);

      \end{tikzpicture}
    \end{centering}
    \caption{If $x$ is NOT in $F$, it simply means that we can find arbitrary big $k$ such that not all of the subcubes $Q'\prec_N Q$ survive. This implies that we can find 
      a porosity hole of relative size $\frac{1}{2} d^{-1/2} M^{-N}$ at this scale and thus $x$ cannot lie in $E^N$.}
      \label{fig:theorem1}
    \end{figure}
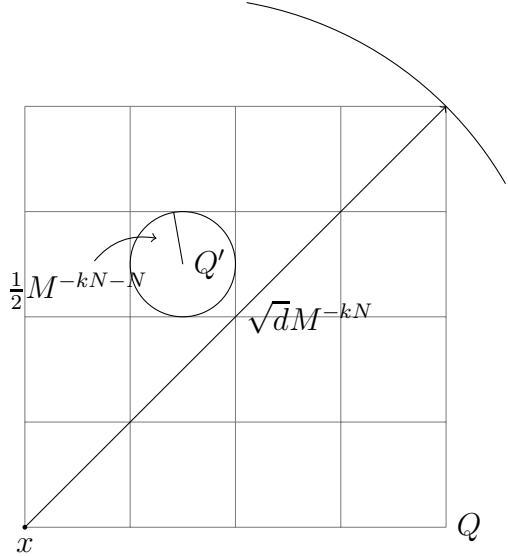
  %%%%%%%%%%%%%%%%%%%%%%%%%%%%%%%%%%%%%%%%%%%%%%%%%%%%%%%%%%%%%%%%%%%%%%%%%%%%%%%%%%%%%%%%%%%%%%%%%%%%%%%%%%%%%%%%
  Denote $E^N:=E^N_{[0,1]^d}$. We will show that $\PP(E^N=\varnothing)=1$ provided $N$ is chosen large enough.
  For a fixed $N\in\N$, we form the following random construction. Set $F_0=[0,1]^d$. Suppose $Q\subset\mathcal {Q}_{Nn}$, $Q\subset F_n$. To construct $F_{n+1}$, we select
  \begin{itemize}
    \item all the sub-cubes $Q'\prec_N Q$ in case all of these are surviving,
    \item and none of them otherwise.
  \end{itemize} Let $F_{n+1}$ be the union of all the selected sub-cubes of cubes $Q\in\mathcal{Q}_{Nn}$ forming $F_n$ and define $F=\cap_{n \in \N} F_n$.
  Now, conditional on $E\neq\varnothing$, the number of cubes in $\mathcal{Q}_{nN}$ forming $F_n$ gives rise to a $GW$-process with offspring distribution 
  $L=M^{dN}\textbf{1}[\widetilde{L}=M^{dN}]$. A simple calculation yields
  \[
    \PP\left(\widetilde{L}=M^{dN}\right)
    =\left( \left(1-q \right)p^{\frac{M^d}{M^d-1}}\right)^{M^{dN}-1},
  \] 
  so that if $N$ is large enough, then $\EE(L)=M^{dN}\PP(\widetilde{L}=M^{dN})<1$ implying (recall \eqref{eq:criticalp}) that $F=\varnothing$ almost surely. But it is clear from the definition of $E^N$ and $F$ that, $E^N\subset F$ always holds (see Figure \ref{fig:theorem1}) 
  so that $E^N=\varnothing$ almost surely. 

    Replacing $[0,1]^d$ by $Q\in\mathcal{Q}_{n_0}$ does not change anything in the argument so we deduce that a.s $E_{Q}^N=\varnothing$ for all $Q\in\mathcal{Q}$. Alternatively, it follows directly from the stochastic self-similarity of the fractal percolation process that $\PP(E_{Q}^N\neq\varnothing)=0$ if and only if $\PP(E^N\neq\varnothing)=0$.

    Finally, since $\underset{x \in E\cap Q}{\inf} \upor(E\cap Q,x)< \alpha$
    is admissible for any $\alpha >0$, the proof is finished by Lemma \ref{lemma:zero-one}.
  \end{proof}

  Let us briefly discuss a variant of Theorem \ref{thm:E_upper_porous} for annular (or spherical) upper porosity. Define the annular upper porosity of a set $A\subset\R^d$ at $x\in A$ as $\upor_a(A,x)=\limsup_{r\rightarrow 0}\por_a(A,x,r)$, where 
  \[\por_a(A,x,r)=\sup\{0\le \varrho<1\,:\,A\cap B(x,r)\setminus B(x,(1-\varrho) r)=\varnothing\}\,.\]
  This is a quantitative notion of total disconnectedness, with applications  in geometric analysis, see e.g. \cite{Tuominen2004}.

  \begin{proposition}\label{prop:annular} 
    There is $p_c>M^{-d}$ such that for all $p \in (M^{-d}, p_c)$, there exist $\kappa=\kappa(d,M,p)>0$, such that a.s. on non-extinction
    $\inf_{x \in E}\upor_a(E,x)= \kappa$. 
  \end{proposition}
  \begin{proof}
    The proof goes along the lines of the proof of Theorem \ref{thm:E_upper_porous}, so we just sketch the idea. Fix $N\in\N$ with $M^N>3+\sqrt{d}$. The main difference to the proof of Theorem \ref{thm:E_upper_porous} is that in constructing $F_{n+1}$, for $Q\in\mathcal{Q}_{nN}$, $Q\subset F_n$ we select
    \begin{itemize}
      \item all the surviving cubes $Q'\prec_N Q$,
      \item except in the case that there is only one such surviving cube $Q'\prec_N Q$ and it satisfies $\dist(Q',\partial Q)>(1+\sqrt{d})M^{-(n+1)N}$, then we do not select it.
    \end{itemize}
    Again, $\#F_n$ is a $GW$-process and the offspring distribution is $L$ with 
    \begin{align*}
      \PP(L=k) =
      \begin{cases}
        \PP(\widetilde{L}=k) \textrm{  when } k\geq 2, \\
        p_b  \textrm{ when } k=1,
      \end{cases}
    \end{align*}
    where $p_b$ is the probability that a uniformly chosen sub-cube of $\mathcal{Q}_{N}$ has distance at least $(1+\sqrt{d})M^{-N}$ to the boundary of $[0,1]^d$. The expectation of $L$ is 
    \begin{align*}\label{expectation}
      \EE(L) &=
      \EE(\widetilde{L})-\PP(\widetilde{L}=1)+p_b\\
      &=(pM^d)^N-(pM^d)^N(1-p+pq)^{(M^d-1)N} +p_b
    \end{align*}
    where $q$ is the probability of extinction. Note that $q \rightarrow 1, p_b \rightarrow 0$ when $p\rightarrow M^{-d}$. It follows that $\EE(L)\rightarrow 0$ when $p \rightarrow M^{-d}$. Thus, for some $p_c>M^{-d}$, we have $\EE(L)<1$ if $p<p_c$ so that $F=\cap_{n\in\N} F_n=\varnothing$ almost surely.

    Since $\underset{x \in E\cap Q}{\inf} \upor_a(E\cap Q,x)< \alpha$
    is admissible for any $\alpha >0$, the claim follows by Lemma \ref{lemma:zero-one}.
  \end{proof}

  \begin{remarks}
    \label{rem:annularremark}

    Recall that a classical result for $d=2$ is the existence of a critical parameter $\widetilde{p}_c=\widetilde{p}_c(M,d)$ such that for $p>\widetilde{p}_c$, there is a positive probability for the existence of left to right crossing path in $E$ while for $p<\widetilde{p}_c$, the limit set $E$ is almost surely totally disconnected (see e.g. \cite{ChayesChayesDurrett1988, Falconer2003, BromanCamia2010}). Proposition \ref{prop:annular} may be seen as a refinement of the latter claim providing a quantitative bound for the total disconnectedness when $p<p_c$. Note that trivially $p_c\le\widetilde{p}_c$. 
  \end{remarks}

  We now turn to the main results dealing with the dimension bounds for the exceptional points where $\upor(E,x)<\tfrac12$ or $\lpor(E,x)>0$.

  \begin{Retheorem}
    \label{thm:Re_upor_dimbound}
    For any $p > M^{-d}$ and $\varepsilon>0$ there exist $\delta =\delta(d,M,p,\varepsilon)>0$, such that a.s.
    \[
      \dim_H\{x\in E: \upor(E,x) < 1/2-\varepsilon\} < s -\delta\,.
    \]
  \end{Retheorem}

  \begin{proof}
    Choose $N=N(\varepsilon)\in\N$ such that $\sqrt{d} M^{-N}<\varepsilon\leq \sqrt{d} M^{-N+1}$. 
    Since we are looking for an upperbound for the dimension of the set where the upper porosity is less that $1/2-\eps$, it suffices to estimate the dimension of a set where this porosity occurs among certain fixed scales. With this in mind, we set up the following notation. For each $n_0\in\N$ and $Q\in\QQ_{n_0 N}$, denote 
    \[
      E_{\varepsilon, Q}=\left\{x\in E\cap Q\,:\,\por\left(E,x,\frac12 M^{-kN}\right)<\frac12-\varepsilon\text{ for all }k\ge n_0\right\}\,.
    \]
    Then $\{x\in E\,|\,\upor(E,x)<\tfrac12-\varepsilon\}
    \subset\cup_{n_0\ge 0}\cup_{Q\in\mathcal{Q}_{n_0 N}}E_{\varepsilon,Q}$ and it suffices to estimate the dimension of each $E_{\varepsilon,Q}$. Denote $E_\varepsilon=E_{\varepsilon, [0,1]^d}$.
    Again, we only show the estimate for $E_{\varepsilon}$ since the general case is similar.
    To that end, we construct a random sequence $F_n\subset E_{nN}\subset [0,1]^d$ as follows: Let $F_0=E_0=[0,1]^d$. Suppose $Q\subset\mathcal {Q}_{Nn}$, $Q\subset F_n$. To construct $F_{n+1}$, we select
    \begin{itemize}
      \item all the surviving cubes $Q'\prec_N Q$,
      \item except that in the case when there is only one such surviving sub-cube of $Q$, then we do not select it.
    \end{itemize}
    Let $F_{n+1}$ be the union of all the selected sub-cubes of cubes $Q\in\mathcal{Q}_{Nn}$ forming $F_n$. 

    Suppose $Q\in\mathcal{Q}_{nN}$ such that $E_\varepsilon\cap Q\neq\varnothing$. Then $Q$ contains at least two disjoint surviving sub-cubes in $\mathcal{Q}_{(n+1)N}$. Indeed, if there was only one cube $Q'\prec_N Q$ surviving, we would have (see Figure \ref{figure2})
    \[
      \por(E, x, \frac{1}{2}M^{-nN})> \frac{1}{2}-\sqrt{d} M^{-N}>\frac12-\varepsilon\,,
    \]
    for $x\in E_\varepsilon\cap Q=E_\varepsilon\cap Q'$ contradicting the definition of $E_\varepsilon$.
    In other words, this means that $E_\varepsilon\subset F$.
    \begin{figure}
      \centering 
      \resizebox{0.4\textwidth}{!}{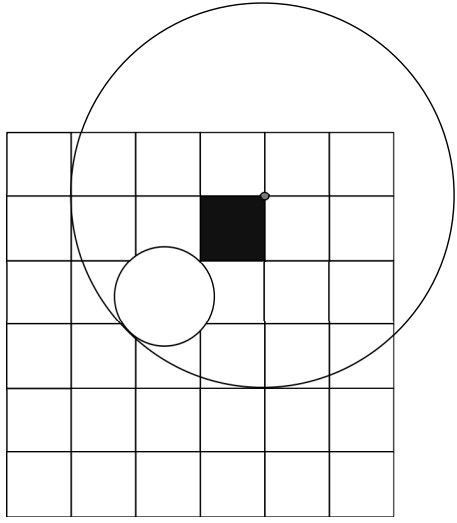}
      \caption{If there is only one cube of level $nN$ surviving, then the porosity at scale $\frac{1}{2}M^{-nN}$ is at least $\frac{1}{2}-\sqrt{d} M^{-N}$.}
      \label{figure2}
    \end{figure}

    To finish the proof, we observe that the construction of $F_n$ fits into the general framework explained in Section \ref{sub:rf}  (where we replace $M$ by $M^N$ throughout) with $L=\widetilde{L}\textbf{1}[\widetilde{L}\ge 2]$. Now $\PP(\widetilde{L}=1)>0$ and $\EE(L)=\EE(\widetilde{L})-\PP(\widetilde{L}=1)$. If $\EE(L)\le 1$, we have $E_\varepsilon\subset F=\varnothing$ almost surely, while if $\EE(L)>1$, then \eqref{eq:dim_formula} yields
    the almost sure estimate
    \[\dim_H(E_\varepsilon)\le\dim_H(F)\le\frac{\log\EE(L)}{\log(M^N)}=\frac{\log\EE(\widetilde{L})}{N\log M}-\delta=s-\delta\,,\]
    for some $\delta=\delta(\varepsilon)>0$. 
  \end{proof}

  \begin{remark}
    We note that a small variant of the above proof yields also the bound
    \[
      \dim_H\{x\in E: \upor_a(E,x) < 1-\varepsilon\} < s -\delta
    \]
    for all $\varepsilon>0$ for the annular porosity considered in Proposition \ref{prop:annular} above. This is obtained by choosing an additional parameter $K=K(\varepsilon)<\infty$ (and $N$ large enough depending on $K$) and changing the definition of $F_{n+1}$ so that we select 
    \begin{itemize}
      \item all the surviving cubes $Q'\prec_N Q$,
      \item except that in the case when there is only one surviving sub-cube of $Q$ and it satisfies $d(Q,\partial Q')>K\sqrt{d}M^{-(n+1)N}$, then we do not select it.
    \end{itemize}
    This will yield the required bound in the same way as in the proof of Theorem \ref{thm:Re_upor_dimbound} above. In particular, we obtain a.s. that $\upor_a(E,x)=1$ for $\mu$-almost all $x\in E$. 
    As being in a connected component implies annular porosity being zero, in particular less than $1- \varepsilon$, we also recover the result of  \cite{BCJM} 
    showing that a.s. the union of all nontrivial connected components of $E$ has Hausdorff dimension $<s-\delta$, for some $\delta=\delta(d,M,p)>0$. 
  \end{remark}

  \begin{Retheorem}

    \label{thm:Re_lpor_dimubound}
    For any $p > M^{-d}$, $\varepsilon>0$ there exist $\delta =\delta(d,M,p,\varepsilon)>0$, such that a.s.
    \[
      \dim_H\{x\in E: \lpor(E,x) > \varepsilon\} < s -\delta.
    \]
  \end{Retheorem}

  \begin{proof}
    Choose $N=N(\varepsilon)\in\N$  such that $3\sqrt{d} M^{-N}<\varepsilon \leq 3\sqrt{d} M^{-N+1}$. 
    For each $Q\in\mathcal{Q}_{n_0N}$, denote 
    \[
      E_{\varepsilon, Q}=\left\{x\in E\cap Q\,:\,\por\left(E,x, \frac{1}{3}M^{-kN}\right)> 3\sqrt{d} M^{-N} \text{ for all }k\ge n_0\right\}.
    \]
    Then 
    $\{x\in E\,:\,\lpor(E,x)>\varepsilon\}
    \subset
    \bigcup_{n_0\ge 0}\bigcup_{Q\in\mathcal{Q}_{n_0 N}}E_{\varepsilon,Q} $
    and it suffices to estimate the dimension of each $E_{\varepsilon,Q}$.
    Denote $E_\varepsilon=E_{\varepsilon, [0,1]^d}$. Again, without loss of generality, it suffices to  
    show that a.s.
    $\dim_H (E_\varepsilon)\le s-\delta(\varepsilon)$. For that purpose, we construct $F_n$, $n\ge 0$ inductively as follows: Let $F_0=[0,1]^d$. If $Q\subset\mathcal {Q}_{Nn}$, $Q\subset F_n$, then to construct $F_{n+1}$,  we select
    \begin{itemize}
      \item all the surviving sub-cubes $Q'\prec_N Q$ if there are strictly less than $M^{Nd}$ of these (i.e. if not all of the sub-cubes are surviving)
      \item and all but the center cube otherwise. 
    \end{itemize}
    If $M$ is odd, it should be clear what we mean by a center cube and if $M$ is even, we just pick any of the cubes in $\mathcal{Q}_{(n+1)N}$ which touches the center of $Q$. Let $F_{n+1}$ be the union of all the selected cubes $Q\in\mathcal{Q}_{(n+1)N}$ and define $F=\cap_{n \in \N} F_n$.  

    We observe that $E_\varepsilon \subset F$. Indeed, if
    $x\in E\setminus F$, then there exists $n\in\N$, $x\in Q'\prec_NQ\in\mathcal{Q}_{nN}$ where $Q'$ is a center cube of $Q$ and all the subcubes of $Q$ in $\mathcal{Q}_{(n+1)N}$ are surviving. This implies the lower bound $\por(E,x,\tfrac13 M^{-nN})\le 3\sqrt{d}M^{-N}$ and whence $x\notin E_\varepsilon$.

    Conditional on $E\neq\varnothing$,  the number of sub-cubes in $\mathcal{Q}_{nN}$ forming $F_n$ is a $GW$-process with offspring distribution $L$ where 
    \[
      L=\begin{cases} \widetilde{L}, &\textrm{ when } \widetilde{L} \neq M^{dN}, \\
        M^{dN}-1, &\textrm{ when } \widetilde{L} = M^{dN}.
      \end{cases}
    \]
    This implies that 
    \[\EE(L)= \EE(\widetilde{L}) - \PP(\widetilde{L}=M^{dN})=s\log(M^N)-\left( \left(1-q \right)p^{\frac{M^d}{M^d-1}}\right)^{M^{dN}-1}\,\] 
    and using \eqref{eq:dim_formula} we obtain 
    $\dimh(F)\le s-\delta$ a.s. for some $\delta= \delta(d,M,p,N)>0$. Recall that we may assume that $\EE(L)>1$ as otherwise $F=\varnothing$ almost surely. Since $E_\varepsilon \subset F$, we arrive at the required upper bound $\dim_H (E_\varepsilon) \leq s-\delta$. 
  \end{proof}

  \begin{Retheorem}
    \label{thm:Re_upor_lowerbound}
    For any $p > M^{-d}$ there exists $\varepsilon_0>0$ such that for all  $\varepsilon\in (0, \varepsilon_0)$ there is $\delta =\delta(d,M,p,\varepsilon)>0$, such that a.s on non-extinction,
    \[\dim_H\{x \in E: \upor(E,x) \leq 1/2-\varepsilon\} > \delta. \]
  \end{Retheorem}

  \begin{proof}
    It is clearly enough to find $\rho<\tfrac12$ and $\delta>0$ such that $\dim_H\{x \in E: \upor(E,x) \leq \rho\} \geq \delta$ a.s on non-extinction.

    Fix $a\in\N$ with $a\ge\sqrt{d}+2$. 
    Given $N\ge a$, we construct $F_n$, $n\ge 0$ inductively as follows: Let $F_0=[0,1]^d$. For $n\ge0$, suppose that $Q\subset\mathcal {Q}_{Nn}$, $Q\subset F_n$ and $Q'\prec_{N-a}Q$ with  $Q'\subset E_{(n+1)N-a}$, here we need $N >a$. To form $F_{n+1}$, we select cubes $Q''\prec_aQ'$ for each such $Q'$ as follows: 
    \begin{itemize} 
      \item  If every cube $Q''\prec_aQ'$ is surviving, then we erase cubes $Q''\prec_a Q'$ with $\dist (Q'', \partial Q')\leq (a-1)M^{-(n+1)N}$ and choose the remaining sub-cubes $Q''\prec_a Q'$. 
      \item Otherwise, we don't  choose sub-cubes of $Q'$.
    \end{itemize}
    Let $F_{n+1}$ be the union of all the selected sub-cubes in $\mathcal{Q}_{(n+1)N}$. 
    Conditional on $E\neq\varnothing$,  the number of sub-cubes in $\mathcal{Q}_{nN}$ forming $F_n$ is a $GW$-process with offspring distribution $L$ that has the same distribution as $\# F_1|\text{non-extinction}$. In particular,
    \[\EE(L)=(1-q)^{-1}\EE(\#F_1)\,.\]  
    Note that the only possible values for $L$ are
    $k(M^a-2a)^d$ for $k = 0, 1, \cdots ,M^{d (N -a)}$,  and that always $M^a-2a >0$. Further,
    \begin{align*}
      &\EE(\#F_1)
      =\EE\left( \sum_{Q\in\QQ_{N-a}} (M^a-2a)^d\textbf{1}[Q\subset E_{N-a}\text{ and all $Q'\prec_a Q$ survive}] \right) \\
      &=(M^a-2a)^d\sum_{Q\in\QQ_{N-a}} \PP\left( Q\subset E_{N-a}\text{ and all $Q'\prec_a Q$ survive}\right) \\
      &=(M^a-2a)^d\sum_{Q\in\QQ_{N-a}} \PP\left( Q\subset E_{N-a}\right) \PP\left(\text{all $Q'\prec_aQ$ survive}\,\mid\,  Q\subset E_{N-a} \right) \\
      &=(M^a-2a)^d \PP\left( \# E_a=M^{ad} \right)(1-q)^{M^{ad}} \sum_{Q\in\QQ_{N-a}} \PP\left( Q\subset E_{N-a}\right) \\
      &=C\EE(\# E_{N-a})=C(d,M,p) (pM^d)^N\,,
    \end{align*}
    where $C$ and $C(d,M,p)$ are positive constant and independent of $N$. For the reason $pM^d > 1$, we may choose large $N$ such that $\EE(L)>1$. Let $F=\cap_{n\in \N} F_n$. Then, by \eqref{eq:dim_formula}, there is a positive probability that
    \[\dim_H(F)=\frac{\log\EE(L)}{\log M^N}=:\delta>0\,.\]

    Now, let $x \in F, 0<r< \sqrt{d}$ and $\rho =\frac{1}{2}
    -\frac{M^{-2N}}{2\sqrt{d}}$. We will show that $\por(E,x,r) \leq \rho$.  Let $n\in\{0,1,\ldots\}$ such that $\sqrt{d}M^{-(n+1)N}\leq r <\sqrt{d}M^{-nN}$ and denote by $Q$ the cube in $\QQ_{(n+1)N}$ that contains $x$ so that $B(x,r)\subset Q$. For any ball $B(y,\rho r) \subset B(x,r)$, we have 
    \begin{equation}\label{hole}
      \dist(x, B(y,\rho r)) \leq r-2\rho r\leq M^{-(n+2)N},
    \end{equation}
    where the last inequality holds since $r <\sqrt{d}M^{-nN}$. 

    Let $L_{x,y}$ be the line segment joining $x$ to the boundary of the ball $B(x,r)$ through $y$. By the construction of $F$, there is a surviving cube $Q'\prec_{N}Q$, with $Q' \cap L_{x,y}\neq \emptyset$ and 
    $\dist(x, Q')\geq (a-1)M^{-(n+2)N}$.
    \begin{figure}
      \centering
      \resizebox{0.6\textwidth}{!}{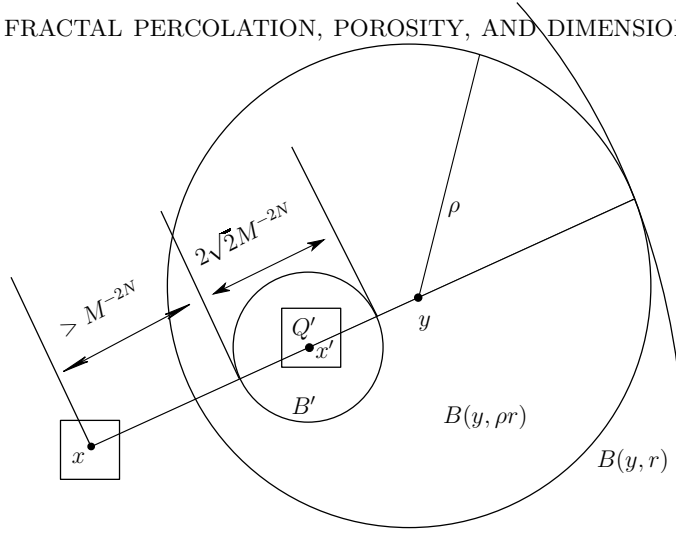}
      \caption{Illustration for the proof of Theorem \ref{thm:upor_lowerbound} in the case $n=0$.
    } 
    \label{thm4} 
  \end{figure}
  Let $x' \in L_{x,y} \cap Q'$ and 
  $B':=B(x', \sqrt{d}M^{-(n+2)N})$. Then (see Figure \ref{thm4})
  \[
  \dist(x, B') \geq \dist(x, Q')-\sqrt{d}M^{-(n+2)N} > M^{-(n+2)N}\,,\]
  by recalling the choice of $a$.
  Together with the estimate \eqref{hole}, 
  we obtain $Q'\subset B' \subset B(x,\rho r)$. Since $Q$ is surviving, we conclude that  
  $B(x,\rho r) \cap E \neq \emptyset$ and whence $\por(E,x,r) \leq \rho$. 
  But this holds for all $x\in F$, $0<r<1$, so we obtain  
  $F \subset \{x \in E\cap\interior[0,1]^d\,:\,\upor(E,x) \leq \rho\}$. Thus there is positive probability $c>0$ that
  \[
    \dim_H\{x \in E\cap\interior[0,1]^d\,:\, \upor(E,x) \leq \rho\} \geq \delta\,.
  \]
  If, instead of $[0,1]^d$, we start with $F_0=Q$ for $Q\in\QQ_n$, and denoting the limit set by $F_Q$, the same proof implies
  \[\PP\left(\dim_H(F_Q)=\delta\,:\,Q\subset E_n\right)\ge c\,.\]
  for all $Q\in\mathcal{Q}_n$ and also that $\upor(E,x)\leq\rho$ for all $x\in F_Q$. Since the event \[\dim_H(F_{Q'})=\delta\text{ for some sub-cube } Q'\subset Q\] 
  is clearly admissible, the proof is finished by Lemma \ref{lemma:zero-one}.
\end{proof}

\begin{Retheorem}

  \label{thm:Re_lpor_lowerbound}
  For any $p > M^{-d}$, there exists $\varepsilon_0>0$ such that for $ \varepsilon\in (0, \varepsilon_0)$ there exist $\delta =\delta(d,M,p,\varepsilon)>0$, such that a.s. on non-extinction,
  \[
    \dim_H\{ x\in E: \lpor(E, x)> \varepsilon\} > \delta\,.
  \]
\end{Retheorem}

\begin{proof}
  Fix $N\in\N$. We construct $F_n$, $n\ge 0$ inductively as follows: Let $F_0=[0,1]^d$. If $Q\subset\mathcal {Q}_{Nn}$, $Q\subset F_n$, then to construct $F_{n+1}$ we select
  \begin{itemize}
    \item all the surviving cubes of $Q'\prec_N Q$ if there are strictly less than $M^{Nd}$ of these (i.e. if not all of the sub-cubes are surviving).
    \item and none of them otherwise.
  \end{itemize}
  and let $F_{n+1}$ be the union of all the selected cubes in $\QQ_{(n+1)N}$.
  Conditional on $E\neq\varnothing$,  the number of sub-cubes in $\mathcal{Q}_{nN}$ forming $F_n$ is a $GW$-process with offspring distribution
  \[
    L=\widetilde{L}\textbf{1}\left[\widetilde{L}<M^{dN}\right]\,.
  \]
  This implies that $\EE(L)=\EE(\widetilde{L})-M^{dN}\PP(\widetilde{L}=M^{dN})$
  Since 
  \[\PP(\widetilde{L}=M^{dN}) =   ((1-q)p^{\frac{M^d}{M^d-1}})^{M^{dN}-1}\,,\] the second term 
  goes to zero as $N$ increases. On the other hand, $\EE(\widetilde{L})=\EE(L_o)^N\rightarrow\infty$ as $N\to\infty$ (recall \eqref{eq:EL}), so we conclude that $\EE(L)>1$ when $N$ is large enough, say $N\ge N_0$. Whence, by \eqref{eq:dim_formula}, there is a positive probability $c>0$ that
  \[
    \dim_H(F)=\frac{\log\EE(L)}{\log M^N}=:\delta>0\,,
  \]
  where $F=\cap_{n\in \N} F_n$. 
  It follows from the construction of $F$, that for each $x\in F$, and $Q\in\mathcal{Q}_{nN}$ with $x\in Q$, there is $Q'\prec_N Q$ with $Q\cap E=\varnothing$. 
  Whence, for all $\sqrt{d}M^{-nN}\le r<\sqrt{d}M^{(1-n)N}$ we have (see Figure \ref{fig:theorem5})  \begin{equation}\label{2Nporo}
    \por(E,x,r)\ge\varepsilon:=\frac12 d^{-1/2} M^{-2N}\,.
  \end{equation}

%%%%%%%%%%%%%%%%%%%%%%%%%%%%%%%%%%%%%%%%%%%%%%%%%%%%%%%%%%%%%%%TIKZ-PICTURE%%%%%%%%%%%%%%%%%%%%%%%%%%%%%%%%%%%%
  \begin{figure}[h]
    \begin{centering}
      \begin{tikzpicture}[scale=.7]
        \filldraw (.4,.4) circle (1pt) node[above] {$x$};
        \draw (.4,.4) -- ++(70:7) node[near start,right,draw=none] {$<\sqrt{d} M^{(1-n) N} $} node (endofrad) {};
        \draw (0,0) rectangle (5,5);
        \node[below] (Q) at (2.5,0)  {$Q$};
        %\draw (0,0) rectangle (1,1);
        \draw (4,4) node[left] {$Q'$} rectangle (5,5);
        \draw (4.5,4.5) circle (.5);
        \draw (4.5,4.5)  -- +(30:.5) node[near start] (smallradius) {};
        \node (teksti) at (7,6) {$\frac{1}{2} M^{-(n+1) N}$};
        \draw[->] (teksti) to[bend left] (smallradius);
        \draw (endofrad) arc (70:90:7);
        \draw (endofrad) arc (70:30:7);
      \end{tikzpicture}
    \end{centering}
    \caption{If not all of the cubes survive, there will be at least one missing and that creates a porosity hole.}
    \label{fig:theorem5}
  \end{figure}
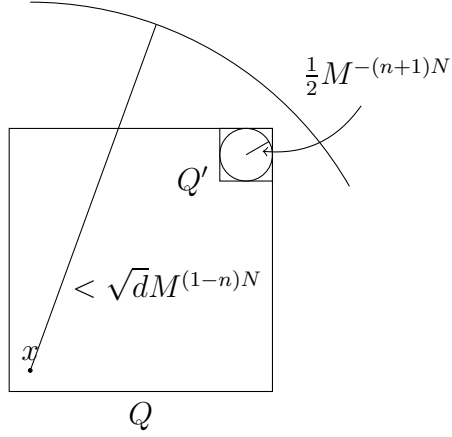
%%%%%%%%%%%%%%%%%%%%%%%%%%%%%%%%%%%%%%%%%%%%%%%%%%%%%%%%%%%%%%%%%%%%%%%%%%%%%%%%%%%%%%%%%%%%%%%%%%%%%%%%%%%%%%%

  Since this holds for all $n\ge 0$, we conclude that $\lpor(E,x)\ge\varepsilon$
  for all $x\in F$ and all $0<r<1$. Further, if instead of $[0,1]^d$, we start the construction of $F$ from $F_0=Q$ for $Q\in\QQ_n$, denoting the limit set by $F_Q$, we have $\PP(\dim_H(F_Q)=\delta\,|\,Q\subset E_n)\ge c$ and likewise $\lpor(E,x)\ge\varepsilon$ for all $x\in F_Q$. 

  The proof is completed by Lemma \ref{lemma:zero-one} by noting that the event \[\dim_H(F_{Q'})\ge\delta \text{ for some sub-cube } Q'\subset Q \] is admissible.
\end{proof}

\begin{remarks}\label{rem:boundaries}
  \emph{a)} If the a.s. dimension is $s>1$, it is actually true that a.s. $\dim_H\{x\in E\,:\,\lpor(E,x)=\tfrac12\}\ge s-1$. This is due to the following fact: Denote by $S$ one of the $M$-adic hyperplanes cutting $[0,1]^d$, say $S=\{x\in\R^d\,:\,x_1=M^{-1}\}\cap\interior[0,1]^d$ and let $U_1,U_2$ be the components of $[0,1]^d\setminus S$. Then there is a positive probability that $\dim_H (E\cap S)=s-1$ and that $E\setminus S$ is completely contained in $U_1$. To see this, we condition on the positive probability event that $E_1\neq\varnothing=E_1\cap U_2$. Then $E\cap S$ is essentially a fractal percolation on $\R^{d-1}$ with the parameters $p$, $M$ and whence $\dim_H(E\cap S)=s-1$ with positive probability. 
%, see Hawkes \cite{Hawkes1981} and 

  Obviously, $\lpor(E,x)=\tfrac12$ for all $x\in E\cap S$ if $E\cap U_1=\varnothing$ and whence there is a positive probability for $\dim_H\{x\in E\,:\,\lpor(E,x)=\tfrac12\}\ge s-1$. To show that this holds with full probability almost surely on non-extinction, one applies Lemma \ref{lemma:zero-one} as in the proofs of Theorems \ref{thm:upor_lowerbound} and \ref{thm:lpor_lowerbound} above.\\

  \emph{b)} Even for $s<1$, it seems reasonable to conjecture that one could choose $\varepsilon_0=\tfrac12$ in Theorem \ref{thm:Re_lpor_lowerbound}. That is, $\dim_H\{x\in E\,:\,\lpor(E,x)>\tfrac12-\varepsilon\}>0$ for all $\varepsilon>0$ a.s. on $E\neq\varnothing$. We note that any set $A\subset\R^d$ with $\lpor(A,x)=\tfrac12$ for all $x\in A$ satisfies $\dimh A\le d-1$ (see \cite{Mattila1988}) and in particular for $d=1$ we have the (deterministic bound) $\dim_H\{x\in E\,:\,\lpor(E,x)=\tfrac12\}=0$.

\end{remarks}

We finish this section by providing the proof of Theorem \ref{thm:Re_asDim_upor_ubound}, which shows that a.s. on non-extinction, the sub- and superlevelsets of upper and lower porosities have the same multifractal structure.

\begin{Retheorem}
  For any $p > M^{-d}, 0\leq \alpha \leq \frac{1}{2}$, there exists $\beta_u=\beta_u(\alpha)$ and $\beta_l=\beta_l(\alpha)$, such that a.s. on non-extinction,
  \[
    \dim_H\{x\in E: \upor(E,x) \leq \alpha\} = \beta_u
    \,\,\text{ and }\,\, 
    \dim_H\{x\in E: \lpor(E,x) \geq \alpha\} = \beta_l.
  \]
  \label{thm:Re_asDim_upor_ubound}
\end{Retheorem}
\begin{proof}
  We start from the first claim.
  For any $0<\alpha<\tfrac12$ and $\beta>0$, the property $\dim_H \left\{ x \in E\cap\interior Q: \upor(E,x) \leq \alpha \right\} > \beta$ is clearly admissible. Furthermore,
  \begin{align*}
    &\PP\left(\dim_H \left\{ x \in E\cap \interior Q: \upor(E,x) \leq \alpha \right\} > \beta\,|\,Q\subset E_n\right)\\
    &=\PP\left(\dim_H \left\{ x \in E\cap \interior [0,1]^d: \upor(E,x) \leq \alpha \right\} >\beta\right)
  \end{align*}
  for all $n\in\N$ and $Q\in\QQ_n$. Thus, Lemma \ref{lemma:zero-one} implies that $\PP(B_{\beta,\alpha}\,|\,\text{non-extinction})\in\{0,1\}$, where $B_{\beta,\alpha}$ is the event $\dim_H \left\{ x \in E: \upor(E,x) \leq \alpha \right\} > \beta$. Note that $\upor(E,x)=\tfrac12$ for all $x\in\partial[0,1]^d$ so that we don't have to worry about the boundary points.
  Trivially, we also have $\PP(A_{\alpha,\beta}\,|\,E\neq\varnothing)=1-\PP(B_{\beta,\alpha}\,|\,E\neq\varnothing)\in \left\{ 0,1 \right\}$, when $A_{\alpha,\beta}$ is the event $\dim_H \left\{ x \in E: \upor(E,x) \leq \alpha \right\} \leq \beta$

  Recall that our goal is to show that for each $\alpha$, there is $\beta_0$ with $\PP\left( C_{\alpha,\beta_0}\,|\,E\neq\varnothing\right) = 1$, where $C_{\alpha,\beta}$ denotes the event $\dim_H \left\{ x \in E: \upor(E,x) \leq \alpha \right\} = \beta$.
  Since $C_{\alpha,\beta} = A_{\alpha,\beta} \cap \bigcap_{n \in \N} B_{\alpha,\beta-\frac{1}{n}}$, we infer that for any $\alpha$ and $\beta$, $\PP(C_{\alpha,\beta}\,|\,E\neq\varnothing) \in \left\{ 0,1 \right\}$. Let $\beta_0 := \inf \left\{ \beta : \PP(A_{\alpha,\beta}\,|\,E\neq\varnothing) =1 \right\}$. Then $\PP(B_{\alpha,\beta_0-\frac{1}{n}}\,|\,E\neq\varnothing)=1$ for every $n \in \N$. Also,  since 
  $A_{\alpha,\beta} = \bigcap_{n\in\N} A_{\alpha,\beta+\frac{1}{n}}$, it follows that $\PP(A_{\alpha,\beta_0}\,|\,E\neq\varnothing) = 1$ and the claim for $\beta_u$ follows.

  The proof for $\beta_l$ is analogous, using that the property 
  \[\dim_H \left\{ x \in E\cap\interior Q: \lpor(E,x) \geq \alpha \right\} > \beta\] is admissible. For points $x\in E\cap\partial[0,1]^d$ one argues as in Remark \ref{rem:boundaries} a) above.
\end{proof}

\begin{remark}\label{rem:porinK} 
  More generally, for any $K\subset[0,\tfrac12]$ the above argument generalizes to show that there is $\gamma_u(K)$, $\gamma_l(K)$ such that a.s. on non-extinction, $\dimh\{x\in E\,:\,\upor(E,x)\in K\}=\gamma_u(K)$, $\dimh\{x\in E\,:\,\lpor(E,x)\in K\}=\gamma_l(K)$. A natural question, which we have not been able to answer is the following: Is it true that $\beta_u(\alpha)=\gamma_u(\{\alpha\})$, $\beta_l(\alpha)=\gamma_l(\{\alpha\})$?
\end{remark}

\section{Generalizations}\label{sec:MoreGeneral}

\subsection{Inhomogeneous fractal percolation}

So far we have considered the porosity properties of (homogeneous) fractal percolation, that is , the model where all the cubes $Q\in\mathcal{Q}_1$ have the same probability $p$ of being retained. In inhomogeneous fractal percolation, we assign a parameter $0\le p_Q\le 1$ to each $Q\in\mathcal{Q}_1$. In the first step of the construction, each $Q\in\mathcal{Q}_1$ is chosen with probability $p_Q$ with all choices independent of each other and the process is repeated independently in each of the retained cubes ad infinitum. It is easy to see that if $p_Q<1$ for at least one $Q\in\mathcal{Q}_1$, the proofs of Theorems \ref{thm:E_upper_porous}, \ref{thm:upor_dimbound}, \ref{thm:upor_lowerbound}, and \ref{thm:lpor_lowerbound} go through in this setting. Furthermore, if $p_Q>0$ for all $Q\in\mathcal{Q}_1$, then also the Theorem \ref{thm:lpor_dimbound} is valid for inhomogeneous fractal percolation.

In the subsection below, we provide extensions of these results and show that the Theorems \ref{thm:E_upper_porous} - \ref{thm:lpor_lowerbound} are actually valid under some rather mild assumptions.

\subsection{Porosity for more general random sets}

In this section we consider the porosity properties of the random sets, as defined in Section \ref{sub:rf}. 
Recall that we have a random sequence of sets $E_n$ formed out of cubes $Q \in \QQ_n$, and the limit set $E=\bigcap_{n \in \N} E_n$ such that the number of cubes forming $E_n$, $\#E_n$, is a $GW$-process. As in the case of fractal percolation, we will denote the generating offspring distribution of this $GW$-process by $L_o$.

Note, that to build the $GW$ processes with offspring distributions $\widetilde{L}$ in the proofs of Theorems \ref{thm:E_upper_porous} - \ref{thm:lpor_lowerbound}, we 
basically only used the information on $L_o$. In the case of homogeneous fractal percolation the distribution of $L_o$ and thus also the distribution of $\widetilde{L}$ (for a fixed $N$) 
are completely determined by the parameter $p$. 
Likewise, in the current situation of more general random sequences $E_n$, we can still form the sequences $F_n$ capturing the porosity properties of $E_n$ such that the number of cubes in $\QQ_{nN}$ forming $F_n$ is again a $GW$ process. 
But instead of being able to explicitly estimate the probability of extinction/non-extinction (and dimension of the limit set) of these processes using the parameter $p$, we must assume something from the initial distribution $L_o$.
Recall that by \eqref{eq:dim_formula}, the set $E$ has constant almost sure dimension, conditioned on non-extinction, 
which we will denote by $s=s(d,M,L_o)$.

We will assume that $\PP(L_o= 1)<1$ and $\PP(L_o= M^d)<1$  
in the  following to exclude trivial cases. 

As before, we denote by $Z_n$ the number of surviving cubes in $\mathcal{Q}_{n}$ and for each $N\in\N$, consider the random variable $\widetilde{L}=\widetilde{L}^{N,M,d,L_o}$ such that
\[\mathbb{P}(\widetilde{L}=k)=\mathbb{P}(Z_N=k\,|\,E\neq\varnothing)\,.\]

For more general random sets, 
we give the following variant of an admissible property by changing 
the random variables $X_{Q'}$ $L_{Q'}$. So, we say that 
a property $\mathcal{A}$ among the pairs $Q\in\QQ$, $E\subset[0,1]^d$ is called \emph{admissible} if
\begin{itemize}
  \item Whether $(Q,E)$ has the property $\mathcal{A}$ is completely determined by the random variables $\{L_{Q'}\}$ for the sub-cubes $Q'\subset Q$ (of all generations) of $Q$.
%set $E\cap \interior Q$, where $\interior$ denotes topological interior.
  \item If $(Q,E)$ has property $\mathcal{A}$, then $(Q',E)$ has property $\mathcal{A}$ whenever $Q\prec_N Q'$ for some $N\in\N$.
\end{itemize}

Lemma \ref{lemma:zero-one} still holds with this weaker definition of admissibility, since the same proof applies 
as it is.

\begin{theorem}
  \label{thm:gen_E_upper_porous}
  There exists $c=c(d,M,L_o)>0$, such that a.s. $\upor(E,x)\geq c$ for all $x\in E$.
\end{theorem}
\begin{proof}
  From $\PP(L_0=M^{d})<1$ it follows 
  that $\PP( Z_N = M^{d N} | E \neq \varnothing ) M^{d N} < 1$ when $N$ is large enough. From this observation, the claim follows as in the proof of Theorem \ref{thm:E_upper_porous}.
\end{proof}
\begin{remark}
  Unlike in the case of the fractal percolation, we cannot conclude that there exists a $c$ for which $\inf_{x \in E} \upor(E,x) = c$ a.s. on non-extinction. For example, one can create a model with $d=2$, $M=4$,
  $\PP(L = 4)=1$ and such that $\PP(E=\{0\}\times[0,1])=1/2$, $\PP(E=C_{1/4})=1/2$, where $C_{1/4}$ is a four corner Cantor set with $\inf_{x \in E} \upor(C_{1/4},x)<1/2$.
\end{remark}

Next we provide a generalization for Theorem \ref{thm:upor_dimbound}. Note, that the proof of Theorem \ref{thm:upor_dimbound} was based on the fact that $\PP(L_o=1)>0$.
Since this is not necessarily true anymore, we have to use a different argument in the proof.

\begin{theorem}
  \label{thm:gen_upor_dimbound}
  Let $\EE(L_o)>1$. Then for any $\varepsilon\in (0, \tfrac12)$ there exist $\delta =\delta(d,M,L_o,\varepsilon)>0$, such that a.s.
  \[
    \dim_H\{x\in E: \upor(E,x) < 1/2-\varepsilon\} < s -\delta.
  \]
\end{theorem}
\begin{proof}
  We may assume that $\PP(0<L_o<M^d)>0$. If this is not true for the original process, it will be true for the process obtained by executing two steps at once.

  Let $N\in\N$ such that $\sqrt{d} M^{-N}<\varepsilon$. We say that a surviving cube $Q'\in\QQ_{(n+1)N}$ meets a porosity hole (see Figure \ref{fig:porosityHole}), if the following holds: Denoting by $Q$ the ancestor of $Q'$ in $\QQ_{nN}$ (that is $Q'\prec_NQ$) there is a cube $Q''\prec_1 Q$ a point $y\in Q''$ and $r\ge\tfrac12 M^{-Nn-1}$ such that $B(y,r)\cap Q'\neq\varnothing$ and $U(y,r)\cap E_{Nn+1}=\varnothing$. Here $U(y,r)$ is the open ball with center $x$ and radius $r$. Note that 
  \begin{equation}\label{eq:porsadas}
    \por\left(E,x,M^{-Nn-1}+\sqrt{d} M^{-(n+1)N}\right)\ge\frac1{2(1+\sqrt{d} M^{-N})}>\frac12-\varepsilon
  \end{equation} 
  for all $x\in Q'$, if $Q'\in\QQ_{N(n+1)}$ meets a porosity hole.
%%%%%%%%%%%%%%%%%%%%%%%%%%%%%%%%%%%%%%%%%%%%%%%%%%%%TIKZ-PICTURE%%%%%%%%%%%%%%%%%%%%%%

  \begin{figure}[h]
    \begin{centering}
      \begin{tikzpicture}
        \draw (0,0) rectangle (8,8) node[near start,left] {$Q$};
        \draw[black!70] (4,4) rectangle (8,8) node[near end,text=black!70] {$Q''$};
        \draw (6,5) circle (2.2);
        \node[below] at (6,5) {$B(y,r)$};
        \draw (6,5) -- ++(30:2.2);
        \draw (6,5)++(-115:2.2) rectangle ++(-1,-1) node[midway] {$Q'$};
      \end{tikzpicture}
    \end{centering}
    \caption{Cube $Q'$ meets a porosity hole. Note that it does not have to be
    a neighboring cube to $Q''$.}
    \label{fig:porosityHole}
  \end{figure}
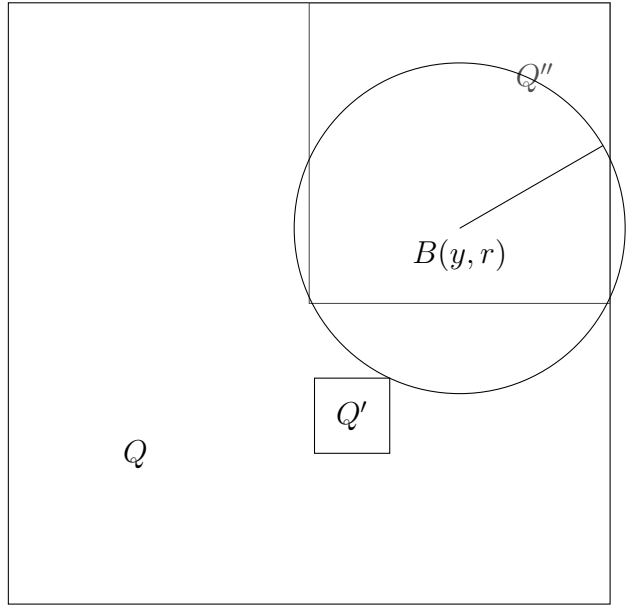

%%%%%%%%%%%%%%%%%%%%%%%%%%%%%%%%%%%%%%%%%%%%%%%%%%%%%%%%%%%%%%%%%%%%%%%%%%%%%%%%%%%%%%%

  Now we follow the argument of the proof of Theorem \ref{thm:upor_dimbound} and 
  construct a random sequence $F_n\subset E_{nN}\subset [0,1]^d$ as follows: 
  Let $F_0=E_0=[0,1]^d$. Suppose $Q\subset\mathcal {Q}_{Nn}$, $Q\subset F_n$. 
  To construct $F_{n+1}$, we select
  \begin{itemize}
    \item all the surviving cubes $Q'\prec_N Q$ if all or none $Q''\prec_1 Q$ are surviving.
    \item otherwise we select all surviving cubes $Q'\prec_N Q$ except one which meets a porosity hole.
  \end{itemize}
  Let $F_{n+1}$ be the union of all the selected cubes  $Q\in\mathcal{Q}_{(n+1)N}$. 

  Note that if $Q\in\QQ_{nN}$ is surviving and not all $Q''\prec_1 Q$ are surviving, then the family of $Q'\prec_N Q$ meeting a porosity hole is nonempty (See Figure
  \ref{fig:porosityHole}), and thus we can indeed remove one of them. It does not really matter which one we remove, but to make the process well defined, we could e.g. remove the first of them in a suitable lexiographic ordering of the sub-cubes. 

  Now, conditional on $E\neq\varnothing$, the number of cubes in $\mathcal{Q}_{nN}$ forming $F_n$ gives 
  rise to a $GW$-process with offspring distribution $L$ such that
  \[
    \EE(L)= \EE(\widetilde{L})-\PP\left(0<L_o<M^d\right)
  \]
  Since $\PP(0<L_o<M^{d})>0$, it is obvious that 
  $\EE(L)<\EE(\widetilde{L})$. Thus, from \eqref{eq:dim_formula}, we have that almost surely,
  \[\dim_H(F)=\max\left\{0,\frac{\log\EE(L)}{\log M^N}\right\}=\frac{\log\EE(\widetilde{L})}{N\log M}-\delta=s-\delta\,,\]
  for some $\delta=\delta(\varepsilon)>0$.

  Finally, recalling \eqref{eq:porsadas}, we note that $E_\varepsilon\subset F$, for 
  \[E_\varepsilon=\{x\in E\,:\,\por(E,x,r)<\tfrac12-\varepsilon\text{ for all }0<r<1\}\,.\]
  so we have actually shown that $\dim_H(E_\varepsilon)<s-\delta$ almost surely. By a similar reasoning, we see that a.s. $\dim_H (E_{\varepsilon,Q})\le s-\delta$, whenever $Q\in\QQ_n$ and
  \[E_{\varepsilon,Q}=\{x\in E\cap Q\,:\,\por(E,x,r)<\tfrac12-\varepsilon\text{ for all }0<r<M^{-n}\}\]
  and so $\dim_H\{x\in E: \upor(E,x) < 1/2-\varepsilon\} \leq s -\delta$. 
\end{proof}

\begin{theorem}
  \label{thm:gen_lpor_dimbound}
  Let $\EE(L_o)>1$ and $\PP(L_0=M^d)>0$. 
  Then for any $\varepsilon>0$ there exist $\delta=\delta(d,M,L_o,\varepsilon)>0$, such that a.s
  \[
    \dim_H\{x\in E: \lpor(E,x) > \varepsilon\} < s -\delta.
  \]
\end{theorem}
\begin{proof}
  As in the proof of Theorem  \ref{thm:lpor_dimbound}, we need to consider the $GW$-process with offspring distribution 
  \[
    L=\begin{cases} \widetilde{L}, &\textrm{ when } \widetilde{L} \neq M^{dN}, \\
      M^{dN}-1, &\textrm{ when } \widetilde{L} = M^{dN}.
    \end{cases}
  \]
  for a suitable $N=N(\varepsilon)\in\N$. Since $\PP(L_o=M^d)>0$, it follows that $\EE(\widetilde{L})<\EE(L)$ and the claim follows exactly as in the proof of Theorem \ref{thm:lpor_dimbound}.
\end{proof}

\begin{remark}
  The assumption $\PP(L_o=M^d)>0$ is clearly necessary for Theorem \ref{thm:gen_lpor_dimbound} as otherwise a.s. $\lpor(E,x)>\tfrac1{2\sqrt{d}M^2}$ for all $x\in E$.
\end{remark}

We will next transfer the results for the lower bounds on the dimension of the exceptional points for the more general setting at hand. We start with a generalization of Theorem \ref{thm:upor_lowerbound}.

\begin{theorem}
  \label{thm:gen_upor_lowerbound}
  Let $\EE(L_o)>1$ and $\PP(L_o > M^{d-1})>0$. Then for any  $\varepsilon \in ( 0, \varepsilon_0(d,M,L_o))$, there is  $\delta=\delta(d, M, L_o, \varepsilon) >0$, such that a.s on non-extinction,
  \[
    \dim_H\{x \in E: \upor(E,x) \leq 1/2-\varepsilon\} > \delta.
  \]
\end{theorem}

In the proof of Theorem \ref{thm:gen_upor_lowerbound}, we will make use of the following discrete conical density lemma. For $x\in\R^d$, $\theta\in S^{d-1}$ and $0<\alpha<1$, we denote
\[H(x,\theta,\alpha)=\{y\in\R^d\,:\,(y-x)\cdot\theta>\alpha|y-x|\}\]
the cone at $x$ into the direction $\theta$ with opening angle $\arccos\alpha\in(0,\pi/2)$. The following lemma is a special case of \cite[Lemma 3.2]{SahlstenShmerkinSuomala2013}.  The lemma holds for any $0<\alpha<1$ although we state it here only for the value $\alpha=\tfrac45$. We note that the condition $\dist(Q_i,Q_j)>M^{-N_0}$ is not included in the statement of the lemma in \cite{SahlstenShmerkinSuomala2013} but it follows directly from the proof. Further, in \cite{SahlstenShmerkinSuomala2013} the lemma is stated for $M=2$ but the proof in the general case is the same.

\begin{lemma}\label{lemma:conical}
  There is $N_0\in\N$ with the property that in any disjoint collection of cubes $\{Q_i\}_{i=1,\ldots,(M^{d-1}+1)^{N_0}}\subset\QQ_{N_0}$, there is a cube $Q_i$ such that for any $x\in Q_i$ and $\theta\in S^{d-1}$, there is $Q_j$ with 
  \begin{equation}\label{good_cubes}
    \dist(Q_j,Q_i)>M^{-N_0}\text{ and }Q_j\subset H(x,\theta,\tfrac45)\,.
  \end{equation}
\end{lemma}
%%%%%%%%%%%%%%%%%%%%%%%%%%%%%%%%%%%%%%%%%%%%%%%%%%%%TIKZ-PICTURE%%%%%%%%%%%%%%%%%%%%%%%%%%%%%%%%%%%%%

\begin{figure}[h]
  \begin{centering}
    \begin{tikzpicture}
      \foreach \i in {1,2,...,40}{
        \pgfmathsetmacro{\x}{(random(64) - 1)/8. }
        \pgfmathsetmacro{\y}{(random(64) - 1)/8. }
        \draw[fill,black!50] (\x,\y) rectangle ++(1/8,1/8);
      }

    %  \draw[fill=white] (3.1,2.2)++(50:6) circle (3) node (y) {};
      \draw (0,0) rectangle (8,8);
      \draw[fill] (3,2.1) rectangle ++(1/8,1/8) node[midway,below] {$Q_i$};
      \draw[fill] (3.1,2.2) circle (1pt);
      \node[above] (x) at (3.1,2.2) {};
     % \draw (3.1,2.2)++(30:9) arc (30:70:9) node[below] (r) {$r$};
      \draw (3.1,2.2)++(50:14) -- +(210:12) node[midway,below=5pt,right]{$B(y,(1/2-\varepsilon) r)$};
      \draw (3.1,2.2)++(50:14)++(200:12) arc (200:255:12) node[above] (rr) {};
    %  \draw[below,fill] (y) circle (1pt) node [above] (texti) {$B(y,(1/2-\varepsilon) r)$};
      \node[right,rotate=50] (teksti) at (x) {$H(x,\theta,4/5) $}; 
      \draw (3.1,2.2) -- +(14:3);
      \draw (3.1,2.2) -- +(86:3);
      \draw[fill] (5.75,4.125) rectangle ++(1/8,1/8) node[midway,below] {$Q_j$};
    \end{tikzpicture}
  \end{centering}
  \caption{Lemma \ref{lemma:conical} (conical density) implies that when we have more than a ``$(d-1)$-dimensional'' family
  of cubes remaining, there will be ``a central'' cube that sees cubes in every direction. Thus, for the
central cube the other cubes form a barrier for the size of a porosity hole.}
\label{fig:conicalDensity}
\end{figure}
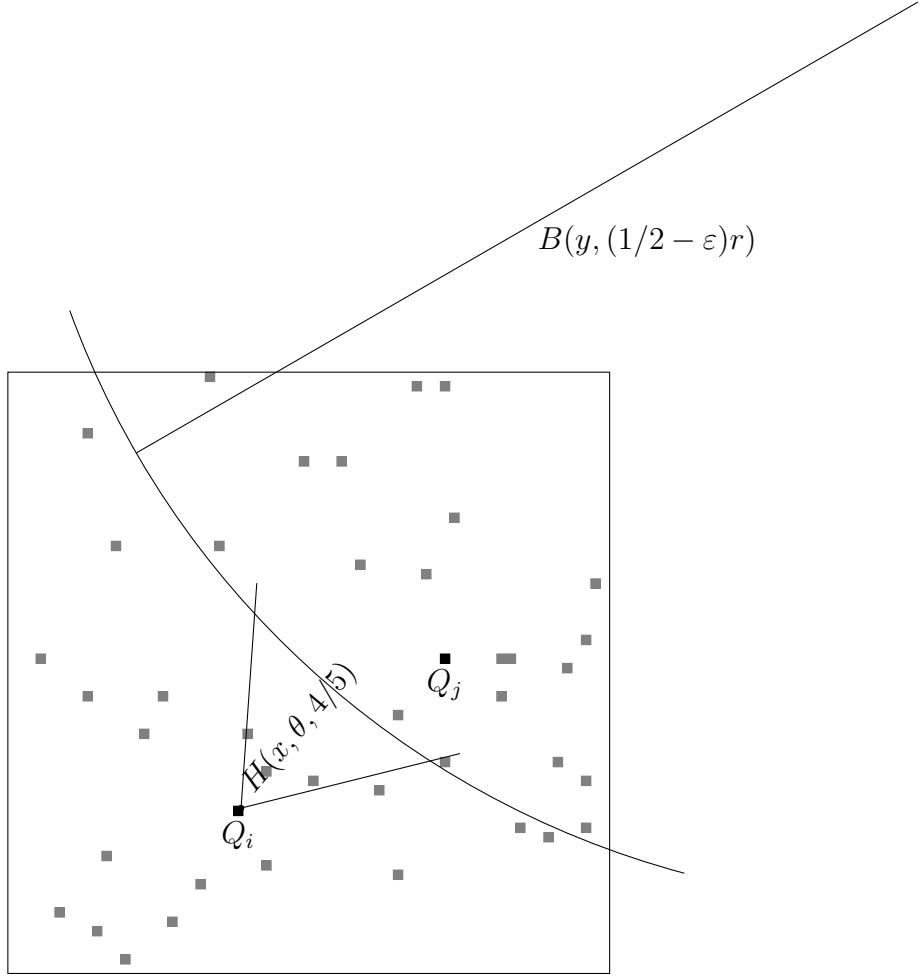

%%%%%%%%%%%%%%%%%%%%%%%%%%%%%%%%%%%%%%%%%%%%%%%%%%%%%%%%%%%%%%%%%%%%%%%%%%%%%%%%%%%%%%%%%%%%%%%%%%%%%
\begin{proof}[Proof of Theorem \ref{thm:gen_upor_lowerbound}]
  Let $N_0$ be the constant of Lemma \ref{lemma:conical} and
  \[p=\PP\left(\text{There are at least } (M^{d-1}+1)^{N_0}\text{ surviving cubes in }\QQ_{N_0}\right)\,.\]
  Note that $p>0$ since $\PP(L_o>M^{d-1})>0$.
  Given $N>N_0$, we construct the following subsets $F_n$ of $E_{nN}$: Let $F_0=[0,1]^d$. If $Q\in\QQ_{nN}$ is contained in $F_n$, we consider all surviving cubes $Q'\prec_{N-N_0} Q$:
  \begin{itemize}
    \item
      if $Q'$ has at least $(M^{d-1}+1)^{N_0}$ surviving sub-cubes $Q_1,\ldots, Q_{(M^{d-1}+1)^{N_0}}\prec_{N_0} Q'$, we apply Lemma \ref{lemma:conical} to this family and select one of them that satisfies the condition \eqref{good_cubes}. 
    \item otherwise, we select no sub-cubes of $Q'$. 
  \end{itemize}
  Let $F_{n+1}$ be the union of the selected cubes in $\QQ_{(n+1)N}$. As in all the previous proofs, the number of cubes forming $F_n$ gives rise to a $GW$-process and if we denote the offspring distribution by $\widetilde{L}$, then
  \[\EE(\widetilde{L})=p\EE(L_o)^{N-N_0}>1\,,\]
  provided $N$ is large enough.

  It remains to show that there is $\varepsilon=\varepsilon(N,N_0)>0$ such that $\upor(E,x)\le\tfrac12-\varepsilon$ for all $x\in F=\cap_{n\in\N}F_n$. To that end, we first recall the following geometric fact. Let $x\in\R^d, \theta \in S^{d-1}$ and $L_{x, \theta}$ be the ray from the point $x$ to the direction $\theta.$ 
  Then there exist $\varepsilon=\varepsilon(d,M,N) >0$, such that for any $B(y, (\tfrac12-\varepsilon)r)\subset B(x,r)$ with $y \in L_{x,\theta}$ and $2\sqrt{d}M^{-N}\leq r<2\sqrt{d}$, we have 
  \begin{equation}\label{geometric_fact}
    H\left(x,\theta,\tfrac45\right) \cap B\left(x,\sqrt{d} M^{-N}\right)\setminus B\left(x,M^{-2N}\right)\subset B\left(y,(\tfrac12-\varepsilon)r\right).
  \end{equation}

  Now, if $x\in F$ and  $2\sqrt{d}M^{-(n+1)N}\leq r<2\sqrt{d}M^{-nN}$ and if 
  \[B\left(y,(\tfrac12-\varepsilon)r\right)\subset B(x,r)\setminus\{x\}\,,\] then by Lemma \ref{lemma:conical} and the definition of $F_n$, there is a surviving cube $Q\in\QQ_{n(N+2)}$ with 
  \[Q\subset H\left(x,\theta,\tfrac45\right)\cap B\left(x,\sqrt{d} M^{-(n+1)N}\right)\setminus B\left(x,M^{-(n+2)N}\right)\,,\] 
  where $\theta=(y-x)/|y-x|$. Combined with \eqref{geometric_fact}, we conclude that $Q\subset B(y,(\tfrac12-\varepsilon)r)$. Whence $E\cap B(y,(\tfrac12-\varepsilon)r)\neq\varnothing$ and consequently $\por(E,x,r)\le\tfrac12-\varepsilon$. Since this holds for all $x\in F$, $0<r<1$, we have shown that with positive probability,
  \[\dimh\{x\in E\,:\,\upor(E,x)\le\tfrac12-\varepsilon\}\ge\dimh F=\delta>0\,.\]
  and the proof is finished using Lemma \ref{lemma:zero-one} as in the proof of Theorem \ref{thm:upor_lowerbound}.
\end{proof}

\begin{remark}
  Under the stronger assumption $\PP(L_o = M^d)>0$, we could use the 'barrier argument'  from the proof of Theorem \ref{thm:upor_lowerbound} instead of Lemma \ref{lemma:conical}.
  The assumption $\PP(L_o>M^{d-1})$ is necessary in the sense that if $L_o\le M^{d-1}$, then it could happen that with positive probability, $E$ is contained in a hyperplane, whence $\upor(E,x)=\tfrac12$ for all $x\in E$.        
\end{remark}

\begin{theorem}
  \label{thm:gen_lpor_lowerbound}
  Let $\EE(L_o)>1$. Then there exist $\varepsilon=\varepsilon(d,M,L_o)$, $\delta=\delta(\varepsilon)>0$, such that a.s. 
  \[
    \dim_H\{ x\in E\,:\,\lpor(E, x)> \varepsilon\} > \delta.
  \]
\end{theorem}
\begin{proof}
  Noting that
  \[M^{dN}\PP(\widetilde{L}=M^{Nd})=M^{dN}\PP(Z_1=M^d)^{(M^{d(N+1)}-M^d)/(M^d-1)}\longrightarrow 0\,,\]
  as $N\to\infty$, we observe that the proof of Theorem \ref{thm:lpor_lowerbound} goes through in the current setting.  
\end{proof}

\section{Further remarks}\label{sec:remarks}

\emph{a)} The proofs of Theorems \ref{thm:upor_dimbound}--\ref{thm:lpor_lowerbound} (and their generalizations in Section \ref{sec:MoreGeneral}) yield  concrete lower bounds for $\delta=\delta(d,M,p,\varepsilon)$ (resp. $\delta(d,M,L_{o},\varepsilon)$). However, these bounds are most likely very far from being optimal. It is a natural question to find the best possible values or at least the correct asymptotics as $\varepsilon\rightarrow 0$.
It is possible to modify the proofs of Theorem \ref{thm:E_upper_porous} and Proposition \ref{prop:annular} to show that $c(d,M,p) \rightarrow 1/2$ and $\kappa(d,M,p)\rightarrow 1$ as $p\rightarrow M^{-d}$. Furthermore,
$c(d,M,p) \rightarrow 0$ as $p \rightarrow 1$ by modifying the argument used in the proof of Theorem \ref{thm:upor_lowerbound}. Recall from Remark \ref{rem:annularremark} that the claim of Proposition \ref{prop:annular} fails for
$p$ close to one.

\emph{b)} In addition to porosity, the method presented in this paper 
may be used to investigate
many other geometric properties of fractal percolation and related models. For instance, we can show that if $p>M^{-d}$, then almost surely on non-extinction, all compact sets $C\subset[0,1]^d$ are $M$-adic 
tangents of $E$ at $x$ for $\mu$-almost all\footnote{Recall that $\mu$ is the natural measure defined in \eqref{eq:mudef}.} $x\in E$ in the sense that given a compact $C\subset[0,1]^d$, 
there is a sequence $n_k\to\infty$ and cubes $x\in Q_k\in\QQ_{n_k}$ such that $\lim_{k\to\infty}h_{Q_k}^{-1}(Q_k(E))=C$ in the Hausdorff metric, where $Q_k(E)$ is the collection of surviving points from $Q_k$.
Formally $Q_k(E) = \left\{ x\in E: \text{ for all } j \text{ there exists } Q \prec_j Q_k \text{ with } x \in Q \subset E_{n_k+j} \right\}$. 

Let us sketch the proof: If $C=Q_1\cup\ldots\cup Q_l$ is a finite union of cubes in $\QQ_{N_0}$, then there is a positive probability for the event
that $E\cap Q_i\neq\varnothing$ for all $i=1,\ldots,l$ and $E\cap Q=\varnothing$ for the rest of the cubes $Q\in\QQ_{N_0}$. Given $Q\in\QQ_k$, denote by $\mathcal{T}(E,Q,N_0)$ the family of cubes in $\QQ_{N_0}$ which intersect $h_{Q}^{-1}(Q( E))$ (recall that $h_Q^{-1}$ is the homothethy mapping $Q$ onto $[0,1]^d$). 
Now we can easily adapt the method from Section \ref{sec:proofs}, to show that almost surely, for all $x\in E$ except a set of dimension $s-\delta(C)<s$, there is a sequence $n_k\to\infty$ and $Q_k\in\QQ_{n_k}$ with $x\in Q_k$ and $\mathcal{T}(E,Q,N_0)=\{Q_1,\ldots,Q_l\}$. A limiting argument using that any compact set may be well approximated with finite unions of cubes then implies that for $\mu$-almost all $x\in E$, all compact sets $C\subset[0,1]$ are tangents in the required sense. Here we have used the fact that the natural measure $\mu$ is almost surely exact dimensional so that the sets of dimension $<s$ are $\mu$-null.

\emph{c)} Our method is heavily based on the use of GW-processes which is due to the nested structure of the random process and some level of stochastic self-similarity. Thus, our method does not work for continuous analogs of the fractal percolation models such as Poissonian cut-out sets and its many variations (see e.g. \cite{Mandelbrot1971, Zahle1984, Falconer2003, NacuWerner2011}). To give a concrete example, let $\Gamma$ be a Poisson point process on $\R^d\times(0,1)$ with intensity $\gamma\mathcal{L}^d\times\tfrac{dr}{r^{d+1}}$ and let $A=B(0,1)\setminus\cup_{(x,r)\in\Gamma}B(x,r)$. For such a random set, it is still natural to conjecture that the claims of Theorems \ref{thm:E_upper_porous}--\ref{thm:lpor_lowerbound} remain valid almost surely on $A\neq\varnothing$, provided $\gamma$ is on the range where the expected dimension is positive.

\emph{d)} Our result suggest that the a.s. quantities $\beta_u(\alpha)$, $\beta_l(\alpha)$ for $0<\alpha<\tfrac12$ possess multifractal behaviour. As mentioned in Remark \ref{rem:porinK}, it would be interesting to know, if the sub- and super level sets in the definition of $\beta_u$, $\beta_l$ could be replaced by the actual level sets $\{x\in E\,:\, \upor(E,x)=\alpha\}$, $\{x\in E\,:\, \lpor(E,x)=\alpha\}$ and if this is the case, analyze the analytical properties of the functions $\alpha\mapsto\beta_u(\alpha)$, $\alpha\mapsto\beta_l(\alpha)$. Note that it follows from Lemma \ref{lemma:zero-one} that
conditional on non-extinction, $\{x\in E\,:\, \upor(E,x)=\alpha\}\neq \emptyset$ has probability either zero or one.

\emph{e) }
In 
Theorems \ref{thm:upor_dimbound}, \ref{thm:lpor_dimbound}, Theorem \ref{thm:gen_upor_dimbound} and Theorem \ref{thm:gen_lpor_dimbound}, the upper bounds obtained for Hausdorff dimension hold. a.s for the packing dimension as well. 
This observation is based on the dimension formula \eqref{eq:boundarydim} (which is valid for Hausdorff, packing- and box-counting dimension) and the countable stability of packing dimension. However we don't know if these 
exceptional sets have exactly the same Hausdorff dimension and packing dimension. Regarding box-counting dimension, it is easy to see that for fractal percolation, if the exceptional 
sets ($\{\upor(E,x)<\alpha\}$ or $\{\lpor(E,x)>\alpha\}$) are non-empty, then they are a.s. dense and whence of full box-counting dimension.

\bibliographystyle{plain}
\bibliography{bibliography}
\end{document}

%% file: thm2.eps_tex
%% Creator: Inkscape 0.48.5, www.inkscape.org
%% Accompanies image file 'thm2.eps' (pdf, eps, ps)
\begingroup%
  \makeatletter%
  \providecommand\color[2][]{%
    \errmessage{(Inkscape) Color is used for the text in Inkscape, but the package 'color.sty' is not loaded}%
    \renewcommand\color[2][]{}%
  }%
  \providecommand\transparent[1]{%
    \errmessage{(Inkscape) Transparency is used (non-zero) for the text in Inkscape, but the package 'transparent.sty' is not loaded}%
    \renewcommand\transparent[1]{}%
  }%
  \providecommand\rotatebox[2]{#2}%
  \ifx\svgwidth\undefined%
    \setlength{\unitlength}{172.55151961bp}%
    \ifx\svgscale\undefined%
      \relax%
    \else%
      \setlength{\unitlength}{\unitlength * \real{\svgscale}}%
    \fi%
  \else%
    \setlength{\unitlength}{\svgwidth}%
  \fi%
  \global\let\svgwidth\undefined%
  \global\let\svgscale\undefined%
  \makeatother%
  \begin{picture}(1,1.14789765)%
    \put(0,0){\includegraphics[width=\unitlength]{thm2.eps}}%
  \end{picture}%
\endgroup%

%% file: thm4.eps_tex
%% Creator: Inkscape 0.48.5, www.inkscape.org
%% Accompanies image file 'thm4.eps' (pdf, eps, ps)
\begingroup%
  \makeatletter%
  \providecommand\color[2][]{%
    \errmessage{(Inkscape) Color is used for the text in Inkscape, but the package 'color.sty' is not loaded}%
    \renewcommand\color[2][]{}%
  }%
  \providecommand\transparent[1]{%
    \errmessage{(Inkscape) Transparency is used (non-zero) for the text in Inkscape, but the package 'transparent.sty' is not loaded}%
    \renewcommand\transparent[1]{}%
  }%
  \providecommand\rotatebox[2]{#2}%
  \ifx\svgwidth\undefined%
    \setlength{\unitlength}{333.14761563bp}%
    \ifx\svgscale\undefined%
      \relax%
    \else%
      \setlength{\unitlength}{\unitlength * \real{\svgscale}}%
    \fi%
  \else%
    \setlength{\unitlength}{\svgwidth}%
  \fi%
  \global\let\svgwidth\undefined%
  \global\let\svgscale\undefined%
  \makeatother%
  \begin{picture}(1,0.67927846)%
    \put(0,0){\includegraphics[width=\unitlength]{thm4.eps}}%
    \put(0.86867628,0.09438054){\makebox(0,0)[lb]{\smash{$B(y, r)$}}}%
    \put(0.09065256,0.09961222){\makebox(0,0)[lb]{\smash{$x$}}}%
    \put(0.0794626,0.28778111){\color[rgb]{0,0,0}\rotatebox{26.46620513}{\makebox(0,0)[lb]{\smash{$> M^{-2N}$}}}}%
    \put(0.28099758,0.39257849){\color[rgb]{0,0,0}\rotatebox{26.2593486}{\makebox(0,0)[lb]{\smash{$2\sqrt{2}M^{-2N}$}}}}%
    \put(0.64775715,0.4721803){\color[rgb]{0,0,0}\makebox(0,0)[lb]{\smash{$\rho$}}}%
    \put(0.60414754,0.30394781){\color[rgb]{0,0,0}\makebox(0,0)[lb]{\smash{$y$}}}%
    \put(0.64111714,0.16005554){\color[rgb]{0,0,0}\makebox(0,0)[lb]{\smash{$B(y,\rho r)$}}}%
    \put(0.41636942,0.17322244){\color[rgb]{0,0,0}\makebox(0,0)[lb]{\smash{$B'$}}}%
    \put(0.41655209,0.2906948){\makebox(0,0)[lb]{\smash{$Q'$}}}%
    \put(0.45191522,0.25551615){\makebox(0,0)[lb]{\smash{$x'$}}}%
  \end{picture}%
\endgroup%